\documentclass[10pt]{amsart}
\usepackage{amsmath,amsfonts,euscript,amscd,amsthm,amssymb,upref,graphics}
\usepackage{mathrsfs,yhmath}
\usepackage[all]{xy}
\usepackage[unicode,pdfencoding=auto,psdextra,hidelinks]{hyperref}
\usepackage{amsrefs}

\calclayout

\usepackage{latexsym}
\usepackage{graphicx}
\usepackage{tikz}
\usepackage{caption}
\usepackage{subcaption}
\usepackage{enumitem}
\usepackage{tikz-cd}

\newcommand{\field}[1]{\mathbb{#1}}

\newcommand{\V}{\field{V}}
\newcommand{\C}{\field{C}}
\newcommand{\CC}{\field{C}}

\newcommand{\N}{\field{N}}

\newcommand{\Q}{\field{Q}}

\newcommand{\Z}{\field{Z}}

\newcommand{\krn}{{\rm ker}\,}

\newcommand{\ins}{N_{n, \tau}}
\newcommand{\tr}{{\rm Tr}}
\DeclareMathOperator{\SL}{SL}
\DeclareMathOperator{\GL}{GL}
\newcommand{\matnn}{\mathrm{Mat}_{n \times n}(\C)}
\newcommand{\matv}{\mathrm{Mat}_{n \times 2}(\C)}
\newcommand{\matw}{\mathrm{Mat}_{2 \times n}(\C)}

\theoremstyle{plain}

\newtheorem{theorem}{Theorem}[section]
\newtheorem{proposition}[theorem]{Proposition}
\newtheorem{lemma}[theorem]{Lemma}
\newtheorem{corollary}[theorem]{Corollary}

\newtheorem{question}[theorem]{Question}

\theoremstyle{definition}
\newtheorem{definition}[theorem]{Definition}
\theoremstyle{remark}
\newtheorem{remark}[theorem]{Remark}
\newtheorem{remarks}[theorem]{Remarks}
\newtheorem{example}[theorem]{Example}

\theoremstyle{definition}


\title[A Criterion for ADP of smooth affine $\mathrm{SL}_2$-varieties]{A Criterion for the algebraic density property of smooth affine $\mathrm{SL}_2$-varieties}

\author[R.B. Andrist]{Rafael B. Andrist}
\address{Faculty of Mathematics and Physics \\
University of Ljubljana \\
Ljubljana, Slovenia}
\email{rafael-benedikt.andrist@fmf.uni-lj.si}

\author[J. Draisma]{Jan Draisma}

\author[G. Freudenburg]{Gene Freudenburg}
\address{Department of Mathematics \\
Western Michigan University}
\email{gene.freudenburg@wmich.edu}

\author[G. Huang]{Gaofeng Huang}
\author[F. Kutzschebauch]{Frank Kutzschebauch}
\address{Mathematisches Institut \\ University of Bern \\  
Bern, Switzerland}
\email{jan.draisma@unibe.ch, gaofeng.huang@unibe.ch, frank.kutzschebauch@unibe.ch}

\subjclass{13N15, 14J60, 14R20}
\keywords{locally nilpotent derivation, $\mathbb{G}_a$-action, $\mathrm{SL}_2$-action, reductive group action, affine variety, density property, Calogero--Moser space} 

\begin{document}

\begin{abstract}
    Let $B$ be an affine $k$-domain which admits a nontrivial fundamental
    pair $(D,U)$ of locally nilpotent derivations, i.e., if $E=[D,U]$
    then $(D,U,E)$ is an $\mathfrak{sl}_2$-triple.  We prove an algebraic
    criterion,
characterizing under which conditions the fundamental pair $(D,U)$
resp.~the triple $(D,U,E)$ is compatible. This compatibility is a technical property that enables us to find  many vector fields on the spectrum of
$B$ from the complete ones. This criterion enables us to prove the algebraic density property for the
following widely studied classes of $\mathrm{SL}_2$-varieties
arising  in physics: Classical Calogero--Moser spaces, Calogero--Moser
spaces with ``inner degrees of freedom'' and a smooth cyclic quiver
variety.
\end{abstract}

\maketitle

\setcounter{tocdepth}{1}
\tableofcontents


\section{Introduction}

The modern study of holomorphic automorphisms of $\CC^n$ started with the seminal papers of Rosay and Rudin \cite{MR0929658} and of Anders{\'e}n and Lempert \cite{MR1185588} and of Forstneri\v{c} and Rosay \cite{MR1213106} in the 1990s.
The \textbf{density property} for Stein manifolds was introduced by Varolin \cite{MR1829353} around 2000 in order to describe what it means for a group of holomorphic automorphisms to be ``very large'' and to generalize many geometric results of $\CC^n$ to a much larger class of complex manifolds.

\begin{definition}[\cite{MR1829353}]
Let $X$ be a Stein manifold. We say that $X$ has the \textbf{density property} if the Lie algebra generated by the complete holomorphic vector fields on $X$ is dense (in the topology of locally uniform convergence) in the Lie algebra of all holomorphic vector fields on $X$.
\end{definition}

In general, it is difficult to verify the density property directly, since one needs to prove a statement about all holomorphic vector fields. It is more convenient to work with a dense subalgebra, for example the algebra of all polynomial vector fields on a smooth affine variety.  For this reason, Varolin also introduced the following, more algebraic notion:

\begin{definition}
\label{def-ADP}
Let $X$ be a smooth affine variety over $\CC$. We say that $X$ has the \textbf{algebraic density property} if the Lie algebra generated by the complete polynomial vector fields on $X$ coincides with the Lie algebra of all polynomial vector fields on $X$.
\end{definition}

By a standard application of Cartan--Serre's Theorem A, the algebraic density property implies the density property.

The main reason for the introduction of the density property by
Varolin is that the Anders{\' e}n--Lempert theorem holds for such
manifolds. This is a Runge approximation theorem for holomorphic
injections by holomorphic automorphisms. With this approximation in
hand, many open problems of geometric nature in several complex
variables have been solved. Clearly, the automorphism group of a Stein
manifold with the density property is huge; in particular, it is infinite-dimensional and acts infinitely transitively. The rapidly developing area of research around the density property is called Anders{\' e}n--Lempert theory. For an overview of applications, examples and techniques, we refer to the recent survey of Forstneri\v{c} and Kutzschebauch \cite{MR4440754}.

A major breakthrough in the construction of examples of Stein
manifolds with the density property was the introduction of the notion of \textbf{compatible pairs} of complete vector fields by Kaliman and Kutzschebauch \cite{MR2385667}.
This enabled them to prove the density property for all complex linear algebraic groups of dimension at least $2$ except for the complex tori $(\CC^\ast)^n, n \ge 2$. Donzelli, Dvorsky and Kaliman \cite{MR2718937} extended the result  to homogeneous spaces $G/H$ where $G$ is a complex linear algebraic group and $H$ is a closed reductive subgroup. Finally, Kaliman and Kutzschebauch \cite{MR3623226} generalized it even further to affine homogeneous spaces of linear algebraic groups $G/H$  without requiring $H$ to be reductive. In all these cases, compatible pairs arising from $\mathrm{SL}_2(\CC)$-actions play a crucial role. 

However, not all smooth affine varieties with the algebraic density property actually admit a compatible pair of algebraic vector fields. 
Therefore, compatible pairs were generalized to compatible $n$-tuples by Andrist, Freudenburg, Huang, Kutzschebauch and Schott \cite{Andrist.Freudenburg.Huang.Kutzschebauch.Schott}. 
This allowed them to simplify many proofs of the density property and to prove the density property for new interesting examples, e.g.,
Gromov-Vaserstein fibers.

\bigskip 

Anders{\' e}n-Lempert theory has not only had success in complex analysis,
it has also proven to be very fruitful in the study of real differentiable
manifolds. As an example, a long-standing problem in 3-dimensional
topology asked whether the fundamental group of any homology 3-sphere
different from the 3-sphere $S^3$ admits an irreducible representation
into $\mathrm{SL}_2(\CC)$. The positive answer given by Zentner
\cite{MR3813594} relied on a real version of Anders{\' e}n-Lempert
theory. Apart from the fact that our methods are purely algebraic, such
applications are an important reason why in the present paper we work
over more general fields than the complex numbers.

Throughout, $k$ denotes a ground field of characteristic zero. An integral domain containing $k$ is a {\bf $k$-domain}.
For the $k$-domain $B$, ${\rm Der}_k(B)$ is the set of $k$-derivations of $B$, and for $D\in {\rm Der}_k(B)$, $\krn D$ denotes its kernel. 
The following generalizes the definition of a compatible $n$-tuple of vector fields from the situation of smooth complex affine varieties given in 
\cite{Andrist.Freudenburg.Huang.Kutzschebauch.Schott} to $k$-domains over a field $k$ of characteristic $0$. 
\begin{definition}
\label{defi-tupel}
Let $B$ be a $k$-domain. Given $n\in\N$, $n\ge 2$, an $n$-tuple 
\[
(\theta_1,\hdots ,\theta_n)\in {\rm Der}_k(B)^n
\]
is {\bf compatible}
if and only if it satisfies each of the following two conditions.
\begin{enumerate}
    \item The subalgebra $k[\krn \theta_1, \krn \theta_2, \dots, \krn \theta_n]$ contains a nonzero $B$-ideal.
    \item There is a graph $(G,\pi ,\epsilon )$ where
    \begin{itemize} 
    \item [(i)] $G$ is a rooted directed tree with orientation toward the root;
    \item [(ii)] $\pi \colon {\rm Vert}(G)\to \{ \theta_1,\hdots ,\theta_n\}$ is a bijection with $\pi ({\rm root}) = \theta_1$;
    \item [(iii)] $\epsilon \colon {\rm Edge}(G)\to B$ is a mapping with
    \[
    \epsilon (v,w)\in \left(\krn\pi (v)^2\setminus\krn\pi (v)\right)\cap \krn \pi (w) \neq \emptyset \, .
    \]
    \end{itemize}
\end{enumerate} 
\end{definition}
Working over $k=\C$, the cited paper studies the case where $B$ is the ring of polynomial functions on a smooth affine variety $X$, and where 
the $n$-tuple $(\theta_1,\hdots ,\theta_n)$ consists of $\C$-complete polynomial vector fields on $X$ which are typically induced by $\CC^+$-actions or $\CC^\ast$-actions.

As explained above, the presence of an $\mathrm{SL}_2$-symmetry was
extremely helpful for showing the density property with the method of
compatible pairs. In the present article we characterize exactly when such a symmetry leads to a compatible pair or a compatible triple. 

We study the case where $B$ is an affine $k$-domain which admits a nontrivial fundamental pair $(D,U)$ of locally nilpotent derivations, i.e., if $E=[D,U]$ then $(D,U,E)$ is an 
$\mathfrak{sl}_2$-triple; see {\it Section\,\ref{fund-pair}}. 
One of our main results is the following algebraic criterion, characterizing under which conditions the fundamental pair $(D, U)$ resp.\ the triple $(D, U, E)$ is compatible. 

\begin{theorem}\label{criterion} 
Let $B$ be a normal affine $k$-domain with a nontrivial fundamental pair $(D,U)$. Set $E=[D,U]$ and let
$A=\krn D$ with induced $\N$-grading $A=\bigoplus_{d\in\N}A_d$ by
$E$-eigenvalues.
\begin{itemize}
    \item [{\bf (a)}] $(E,D,U)$ is a compatible triple for $B$ if and only if $A_2\ne\{ 0\}$.
    \item [{\bf (b)}] $(D,U)$ is a compatible pair for $B$ if and only if $A_1\ne \{ 0\}$. 
    \end{itemize}
\end{theorem}
The proof for this criterion is given in {\it Section\,\ref{general-case}}. 
Note that the tree for the compatible triple in part (a) is $D\rightarrow E\leftarrow U$, and for the compatible pair in part (b) is $U\rightarrow D$. 

In {\it Section\,\ref{applications}} this criterion is applied to the case of a smooth affine $\C$-variety $X$
which admits a nontrivial algebraic $\mathrm{SL}_2$-action, see Theorem \ref{criterionADP}. In this case, the existence of such an action is equivalent to the existence of a nontrivial fundamental pair $(D,U)$ for $\C [X]$ (the ring of regular functions on $X$),
and the vector fields on $X$ determined by 
$D,U,E\in {\rm Der}_{\C}(\C [X])$ are complete. 

We then apply Theorem \ref{criterionADP} to prove the algebraic density property for the following widely studied classes of $\mathrm{SL}_2$-varieties: Classical Calogero--Moser spaces, Calogero--Moser spaces with ``inner degrees of freedom'' and one smooth cyclic quiver variety. 

Calogero--Moser spaces originate from classical physics, where they describe the completed phase space of a finite fixed number of classical, indistinguishable particles with a certain potential. However, Calogero--Moser spaces allow for ``collisions'' of particles, i.e.\ they do not necessarily need to have different spatial coordinates. The density property for classical Calogero--Moser spaces has been studied before by Andrist \cite{MR4305975}.

\section{Preliminaries}

\subsection{Conductor ideal}
Let $R\subset S$ be integral domains. The {\bf conductor} of $S$ in $R$ is:
\[
\mathcal{C}_R(S)=\{ r\in R\, |\, rS\subset R\}
\]
This is an ideal of both $R$ and $S$, and is the largest $S$-ideal contained in $R$.

\begin{lemma}\label{conductor} {\cite{MR3700208}*{ Lemma 1.21}} Let $K$ be a field, $A$ a $K$-domain and $\mathcal{L}_A$ the integral closure of $A$ in ${\rm frac}(A)$. 
If $A$ is $K$-affine then $\mathcal{C}_A(\mathcal{L}_A)\ne (0)$.
\end{lemma}

\subsection{Locally nilpotent derivations} We describe some of the basic properties of locally nilpotent derivations used in subsequent sections. The reader is referred to \cite{MR3700208} for a more extended treatment of the subject.

Let $B$ be a $k$-domain. Given $D\in {\rm Der}_k(B)$, $D$ is {\bf
locally nilpotent} if, for each nonzero $f\in B$, $D^nf=0$ for $n\gg
0$. Define the degree function $\deg_D \colon B\to \N\cup\{ -\infty\}$
by $\deg_Df=\max\{ n\in\N\, |\, D^nf\ne 0\}$ for $f\in B\setminus\{
0\}$ and $\deg_D0=-\infty$. The set of all locally nilpotent derivations of $B$ is 
${\rm LND}(B)$.  

Given $D\in {\rm LND}(B)$, let $A=\krn D$. 
The {\bf plinth ideal} for $D$ is the $A$-ideal ${\rm pl}(D)=A\cap DB$. 
Any $g\in B$ with $\deg_D(g)=1$ is called a {\bf local slice} for $B$. If $Dg=1$ then $g$ is called a {\bf slice}. 

\begin{theorem} {\rm (Slice Theorem)} \label{thm: slice} If $D\in {\rm
LND}(B)$ and $Dg=1$ for some $g\in B$, then $B=A[g]\cong_kA^{[1]}$,
a polynomial ring over $A$ in one variable.
\end{theorem}

For any multiplicatively closed set $S\subset A \setminus\{ 0\}$, $D$ extends to a locally nilpotent derivation on the localization $S^{-1}B$ by $D(b/s)=Db/s$. Suppose that $g\in B$ is a local slice for $D$, and set $f=Dg\in A$. Extend $D$ to the localization $B_f$. Then $D(g/f)=1$ and the Slice Theorem implies:
\[
B_f=A_f[g/f]=A_f[g]\cong_kA_f^{[1]}
\]
where $A_f^{[1]}$ is the polynomial algebra in one variable over $A_f$. 

The Slice Theorem immediately implies the following. 
\begin{lemma}\label{useful} 
Let $B$ be a $k$-domain, $D$ a locally nilpotent derivation of $B$, $A = \krn (D)$ and $r$ a local slice of $D$. Then  ${\rm frac}(A[r])={\rm frac}(B)$.
\end{lemma}


\begin{lemma}\label{lin-independent} Let $B$ be a $k$-domain and $D\in {\rm LND}(B)$ nonzero. Given $f\in B$, $f\ne 0$, set $d=\deg_D(f)$. Then $f,Df,\hdots ,D^df$ are linearly independent over $k$.
\end{lemma}

\begin{proof} Suppose that $\sum_{i=0}^dc_iD^if=0$ for $c_i\in k$.
Applying $D^d$ yields $c_0D^df=0$. Since $D^df\ne 0$ we see that $c_0=0$. Next, apply $D^{d-1}$ to get $c_1=0$. Continuing in this way, we see that $c_i=0$ for each $i$. 
\end{proof} 

\subsection{Fundamental pairs}\label{fund-pair}
Let $B$ be a $k$-domain. Assume that $D,U\in {\rm LND}(B)$ satisfy the relations:
\[
[D,[D,U]]=-2D \quad {\rm and}\quad [U,[D,U]]=2U
\]
Such a pair $(D,U)$ is called a {\bf fundamental pair} for $B$. 
Define the $k$-derivation $E=[D,U]$. For each $d\in\Z$ we have $E-dI\in {\rm End}_k(B)$. Define 
\[
B_d=\krn (E-dI)\,\, ,\,\, d\in\Z
\]
noting that $B_d$ is a $B_0$-module. 

Set $A=\krn D$ and $\Omega =\krn U$. The {\bf degree modules} for $D$ are $\mathcal{F}_n=\krn D^{n+1}$, $n\in\N\cup\{ -1\}$, and the {\bf image ideals} for $D$ are $I_n=A\cap D^nB$, $n\in\N$. 
The kernel $\mathcal{F}_n$ is an $A$-module, and $I_n$ is an ideal of $A$. An ideal $J\subset B$ is {\bf $(D,U)$-invariant} if $DJ\subset J$ and $UJ\subset J$. 

The main results of this section are given in the following three theorems. 
The first describes the structure of $B$ determined by $(D,U)$;
the second describes the structure of $A$ determined by $(D,U)$;
and the third describes the behavior of $(D,U)$-invariant ideals.
Note that, for these results, we do not assume that $B$ is affine. 

\begin{theorem}\label{main1} The following properties hold. 
\begin{itemize}
\item [{\bf (a)}] $\krn U^n\cap {\rm im}\,D^n=\{ 0\}$ for each $n\ge 0$. 
\item [{\bf (b)}]  As $A_0$-modules:
\[
B=\bigoplus_{i\ge 0}U^iD^i(\mathcal{F}_i) \quad {\rm and}\quad \mathcal{F}_n=\mathcal{F}_{n-1}\oplus U^nD^n(\mathcal{F}_n) \,\, ,\,\, n\ge 0
\]
\item [{\bf (c)}] $E$ is semisimple as a $k$-derivation of $B$, and $B=\bigoplus_{d\in\Z}B_d$ is a $\Z$-grading.
\end{itemize}
\end{theorem}

\begin{theorem}\label{main2} 
Given $d\in\N$ let $A_d=A\cap B_d$. The following properties hold. 
\begin{itemize}
\item [{\bf (a)}] $A=\bigoplus_{d\ge 0}A_d$ is an $\N$-grading.
\item [{\bf (b)}] Given nonzero $f\in A_d$, $d=\deg_Uf$.
\item [{\bf (c)}] Given $n\in\N$, $I_n=\bigoplus_{d\ge n}A_d$. 
\end{itemize}
\end{theorem}

\begin{theorem}\label{main3}
Let $\mathfrak{p}\subset B$ be a graded (see Theorem \ref{main1}(c)) and $(D,U)$-invariant ideal. 
\begin{itemize}
\item [{\bf (a)}] $D^{-1}(\mathfrak{p})=A+\mathfrak{p}$
\item [{\bf (b)}] Assume that $\mathfrak{p}$ is a prime ideal. 
Let $\pi \colon B\to B/\mathfrak{p}$ be the standard surjection and let $(\pi (D),\pi (U))$ be the induced fundamental pair on $B/\mathfrak{p}$.
Then $\krn (\pi (D))=\pi (\krn D)$ and $\krn (\pi (U))=\pi (\krn U)$.
\end{itemize}
\end{theorem}
We also need the following. 

\begin{lemma}\label{useful2} 
Define subrings $R=k[\krn D,\krn U]$ and $S=R[\krn E]$.
\begin{itemize}
\item [{\bf (a)}] If $A_1\ne\{ 0\}$ 
then ${\rm frac}(R)={\rm frac}(B)$.
\item [{\bf (b)}] If $A_2\ne\{ 0\}$ 
then ${\rm frac}(S)={\rm frac}(B)$.
\item [{\bf (c)}] $A_1\ne\{ 0\}$ if and only if there exists $g\in B$ with $\deg_U(g)=0$ and $\deg_D(g)=1$. 
\item [{\bf (d)}] $A_2\ne\{ 0\}$ if and only if there exists $g\in B$ with $\deg_D(g)=\deg_U(g)=1$.  
\end{itemize} 
\end{lemma}

Proofs for these four results are given in {\it Appendix\,\ref{proof-fund-pair}}.

Fundamental pairs $(D,U)$ were studied in \cite{Freudenburg2022}, where the definition included the hypothesis that $[D,U]$ is semi-simple, 
and asked whether this hypothesis could be removed. {\it Theorem\,\ref{main1}} shows that, indeed, this hypothesis is unnecessary, i.e., if $(D,U)$ satisfies the given Lie relations then $[D,U]$ is semi-simple. 
In \cite{MR3049288}*{Proposition 2.1}, Arzhantsev and Liendo give the following criterion for the existence of an $\mathrm{SL}_2$-action on an affine variety over an algebraically closed field of characteristic zero, describing it as ``well-known''.
\begin{quote}{\it
A nontrivial $\mathrm{SL}_2$-action on an affine variety $X = { Spec}(A)$ is
equivalent to a (not necessarily effective) $\Z$--grading on $A$ with infinitesimal generator
$\delta$ and a couple of homogeneous locally nilpotent derivations $(\delta_+,\delta_-)$ of $A$ of degrees $\deg_{\Z}\delta_{\pm} = \pm 2$
and satisfying $[\delta_+,\delta_-]=\delta$.}
\end{quote}
Consequently, when $k$ is algebraically closed, $\mathrm{SL}_2$-actions on $X$ (or $A$) are equivalent to fundamental pairs $(D,U)$ for $A$, i.e., $D$ and $U$ are locally nilpotent and $(D,U,[D,U])$ is an $\mathfrak{sl}_2$-triple. 

\section{Basic Fundamental Pairs on Polynomial Rings}
Given the integer $d\ge 1$ let $B=k[x_0,\hdots ,x_d]=k^{[d+1]}$. 
The {\bf basic} fundamental pair $(D,U)$ on $B$ is defined by:
\[
Dx_i=x_{i-1} \,\, (1\le i\le d) \quad \text{and}\quad Dx_0=0
\]
and:
\[
Ux_i=(i+1)(d-i)x_{i+1} \,\, (0\le i\le d-1) \quad \text{and}\quad Ux_d=0
\]
This pair defines the $\mathrm{SL}_2$-action on $B$ induced by the unique (up to the action of $\GL_{d+1}(k)$) irreducible representation of $\mathrm{SL}_2$ on $\V_d:=kx_0\oplus\cdots\oplus kx_d=k^{d+1}$;
see \cite{Freudenburg2022}. Set $A=\krn D$ and $\Omega=\krn U$, and $E=[D,U]$. 

Write $d=2m+\delta$ for $\delta\in\{ 0,1\}$ depending on whether $d$
is even or odd. It is well-known that the vector space of quadratic
elements of $A=\krn(D)$ has basis 
$\{ T_0,T_2,T_4,\hdots ,T_{2m}\}$ where:
\[
T_{2i}=\sum_{0\le j\le 2i}(-1)^jx_jx_{2i-j} \quad (0\le i\le m)
\]
In particular, the coefficient of $x_i^2$ in $T_{2i}$ is $(-1)^i$, $0\le i\le m$. 
Define the involution $\alpha\in \GL_{d+1}(k)$ by:
\[ 
\textstyle\alpha (x_i)=\frac{(d-i)!}{i!}\, x_{d-i} \quad (0\le i\le d)
\]
Then $\alpha (D,U)\alpha^{-1}=(U,D)$. Consequently, 
the vector space of quadratic elements of $\Omega$ has basis $\{ \alpha (T_0),\alpha (T_2),\hdots ,\alpha (T_{2m})\}$.
In particular, the coefficient of $x_{d-i}^2$ in $\alpha (T_{2i})$ is  $(-1)^i(\frac{(d-i)!}{i!})^2\ne 0$, $d-m\le d-i\le d$. 
We make the following definitions. 
\begin{enumerate}
\item $R=k[t_0,\hdots ,t_d]\cong_kk^{[d+1]}$. 
\item $K=k[x_0^2,\hdots ,x_d^2]\cong_kR$. 
\item $\beta \colon R\to B, t_j \mapsto x_j^2$ is the inclusion with image $K$.
\item $(B,\beta )$ denotes $B$ with the $R$-module structure defined by $\beta$. 
\item Given $0\le i\le m$ let $y_i=T_{2i}$ and $y_{d-i}=\alpha (T_{2i})$, and set 
    \[ S=k[y_0,\hdots ,y_d] \cong_k R \]
\item $\gamma \colon R \to B, t_j \mapsto y_j$ is the inclusion with image $S$.
\item $(B,\gamma )$ denotes $B$ with the $R$-module structure defined by $\gamma$. 
\item $X\subset B$ is the (finite) set of monic square-free monomials in $x_0,\hdots ,x_d$. 
\end{enumerate}
\begin{proposition}\label{free-module1} $(B,\beta )$ is a free $R$-module with basis $X$. 
\end{proposition}

\begin{proof} Let $G=\Z_2^{d+1}$ with basis $e_i=(\delta_{i,0},\hdots ,\delta_{i,d})$, $0\le i\le d $, where $\delta_{i,j}$ is the Kronecker delta.
Define an action of $G$ on $B$ by $e_i\cdot (x_j)=(-1)^{\delta_{i,j}}x_j$. Then
$B^G=k[x_0^2,\hdots ,x_d^2]\cong k^{[d+1]}$. Since $G$ is a finite group, $B$ is an integral extension of $B^G$; see \cite{MR3700208}*{\S 6, Appendix 1}. 

It is clear that $B$ is spanned by $X$ over $B^G$. 
Suppose that $r_1m_1+\cdots +r_nm_n=0$ for some $n\in\Z$, $n\ge 2$, $r_i\in B^G$ not all zero, and distinct $m_i\in X$. We may assume that $r_1,\hdots ,r_n$ have no common factor in $B^G$. 
Let $\pi \colon B\to B/(x_d^2)$ be the natural surjection. We have:
\[
B/(x_d^2)=\tilde{B}\oplus \tilde{B}x_d \quad \text{where}\quad \tilde{B}=k[x_0,\hdots ,x_{d-1}]
\]
Note that $\pi$ is injective on $X$. Let $\tilde{X}=X\cap\tilde{B}$. We may assume that, for some $0\le l\le n$,  $m_1,\hdots ,m_l\in \tilde{X}$ and $m_{l+1},\hdots ,m_n\in\tilde{X}x_d$. 
Let $\tilde{m}_j=m_j/x_d$ for $l+1\le j\le n$.
Then:
\[
(r_1m_1+\cdots +r_lm_l)+(r_{l+1}\tilde{m}_{l+1}+\cdots +r_n\tilde{m}_n)x_d=0
\]
Therefore
\[
(\pi (r_1)m_1+\cdots +\pi (r_l)m_l)+(\pi (r_{l+1})\tilde{m}_{l+1}+\cdots +\pi (r_n)\tilde{m}_n)x_d=0
\]
which implies:
\[
\pi (r_1)m_1+\cdots +\pi (r_l)m_l=0 \quad\text{and}\quad \pi (r_{l+1})\tilde{m}_{l+1}+\cdots +\pi (r_n)\tilde{m}_n=0
\]
By induction on $d\ge 1$ (the case $d=1$ being clear) we have $\pi (r_i)=0$ for $1\le i\le n$. But then $r_1,\hdots ,r_n$ have a common factor $x_d^2$, giving a contradiction. 
Therefore, $r_i=0$ for each $i$, and $X$ is a basis for $B$ as a $B^G$-module. 
\end{proof} 

\begin{proposition}\label{free-module2}
$(B,\gamma )$ is a free $R$-module with basis $X$. 
\end{proposition}

\begin{proof} 
The isomorphism $\theta \colon K \to S$ defined by $\theta (x_i^2)=y_i$ induces an $R$-module isomorphism of $\bigoplus_{m\in X}\beta (R)m$ with $\bigoplus_{m\in X}\gamma (R)m$. 
By {\it Proposition\,\ref{free-module1}}, $(B,\beta )=\bigoplus_{m\in X}\beta (R)m$.  
It thus suffices to show that $B$ is spanned by $X$ as an $S$-module. 

Given $y_i$, $0\le i\le d$, there exists $t\in\{0,\hdots ,m\}$ such that $y_i$ is a linear combination of monomials $x_{t-j}x_{t+j}$, $0\le j\le t$. 
Choose a term order for monomials in $B$ so that, for each $t$:
\[
x_t^2>x_{t-1}x_{t+1}>\cdots >x_{t-j}x_{t+j}>\cdots >x_0x_{2t}
\]
For example, order the variables by $x_d>x_{d-1}>\cdots >x_1>x_0$
and consider the degree reverse lexicographical (degrevlex) order.
In this order, $x_t^2$ is the leading term of $y_i$ (up to multiplication by a nonzero constant).

Let $M\subseteq B$ be the $S$-module generated by $X$. 
If $M\ne B$, then choose $f\in B\setminus M$
with minimal leading monomial $u$. We may assume that the coefficient of $u$ is 1. 
Write $u=m^2 u'$ where $m,u'$ are monomials and $u'$ is square-free. There is an
element $g\in S$ whose leading monomial is $m^2$ (just take $m$ and replace $x_i$ by $y_i$, $0\le i\le d$). Subtracting $g u'$ from
$f$ yields a contradiction. Therefore, we must have $M=B$.
\end{proof}

\begin{proposition}\label{A1-A2}
Let $d\ge 1$ be given. 
\begin{itemize}
\item [{\bf (a)}] If $d$ is odd and $d\ne 3$ then $A_1\ne\{ 0\}$.
\item [{\bf (b)}] If $d\ne 4$ then $A_2\ne\{ 0\}$. 
\end{itemize}
\end{proposition}

\begin{proof}
Given $i,j\in\N$ let $A_{(i,j)}$ denote the set of elements of $A_i$ of standard degree $j$ in $x_0,\hdots ,x_d$. Also, let $B'=k[\V_5]$ and with fundamental pair $(D',U')$ and $A'=\krn D'=\bigoplus_{e\in\N}A_e^{\prime}$, and let $A_{(i,j)}^{\prime}$ be the set of elements of $A_i^{\prime}$ of standard degree $j$. Then $A_{(i,5)}$ and $A_{(i,d)}^{\prime}$ are finite dimensional vector spaces over $k$, and by Hermite reciprocity, we have:
\[
\dim A_{(i,5)}=\dim A_{(i,d)}^{\prime}
\]
See \cite{MR0974333}*{\S 3.3.2.2}. The groundforms of $A^{\prime}$ are well known. 
In particular, $A_{(1,5)}^{\prime}\ne \{ 0\}$ and $A_{(1,7)}^{\prime}\ne \{ 0\}$, and there exists a nonzero element of $A_0^{\prime}$ of standard degree 4; see \cite{MR3700208}*{\S 6, Appendix 2}. 
Therefore, $A_{(1,d)}^{\prime}\ne\{ 0\}$ for all odd $d\ge 5$. This implies $A_{(1,5)}\ne \{ 0\}$, so 
$A_1\ne \{ 0\}$ when $d$ is odd, $d\ge 5$. Since $A_1\ne \{ 0\}$
for $d=1$ as well, part (a) is confirmed. 

Similarly, $A_{(2,e)}^{\prime}\ne \{ 0\}$ for $e\in\{ 2,6,8\}$, which 
implies $A_{(2,d)}^{\prime}\ne \{ 0\}$ for all even $d\ge 2$, $d\ne 4$. By reciprocity, we have $A_{(2,5)}\ne \{ 0\}$, hence $A_2\ne\{ 0\}$, for all even $d\ge 2$, $d\ne 4$. For odd $d\ne 3$, part (a) shows $A_2\ne \{ 0\}$, and for $d=3$ we already know $A_2\ne \{ 0\}$ (the polynomial $T_2 = 2 x_0 x_2 - x_1^2 \in A_2$ cf.\ \cite{MR3700208}*{\S 8.7.1}). So part (b) is confirmed.
\end{proof}

\begin{corollary}\label{fraction-field} 
Given $d\ge 1$ define subalgebras of $k[\V_d]$:
\[
\mathcal{R}=k[\krn D,\krn U] \quad\text{and}\quad 
\mathcal{S}=\mathcal{R}[\krn E]
\]
\begin{itemize}
\item [{\bf (a)}] ${\rm frac}(\mathcal{S})={\rm frac}(B)$
\item [{\bf (b)}] If $d$ is odd then 
${\rm frac}(\mathcal{R})={\rm frac}(B)$.
\end{itemize}
\end{corollary} 

\begin{proof} 
Part (a). 
If $d\ne 4$ then part (a) follows by {\it Proposition\,\ref{A1-A2}(b)} and {\it Lemma\,\ref{useful2}(b)}. 
If $d=4$ there exist nonzero $f\in A_4$ and $g\in A_6$; see \cite{MR3700208}*{\S 8.7.1}. 
Therefore:
\[
x_3^2f ,x_3^3g\in B_0 = \ker E \subset\mathcal{S} \implies 
x_3^2,x_3^3\in {\rm frac}(\mathcal{S}) \implies
x_3\in {\rm frac}(\mathcal{S})
\]
By symmetry, $x_1\in{\rm frac}(\mathcal{S})$, and we see that
${\rm frac}(\mathcal{S})={\rm frac}(B)$. 

Part (b). If $d\ne 3$ this follows from {\it Proposition\,\ref{A1-A2}(a)} and {\it Lemma\,\ref{useful2}(a)}. 

Consider the case $d=3$. We have $A=k[x_0,f,g,h]$ and 
$\Omega =k[x_3,F,G,h]$ where:
\begin{eqnarray*}
f&=&2x_0x_2-x_1^2 \\
g&=&3x_0^2x_3-3x_0x_1x_2+x_1^3 \\
F&=&3x_1x_3-2x_2^2 \\
G&=&3x_0x_3^2-3x_1x_2x_3+\textstyle\frac{4}{3}x_2^3 \\
h&=&9x_0^2x_3^2-18x_0x_1x_2x_3+8x_0x_2^3+6x_1^3 x_3-3x_1^2x_2^2
\end{eqnarray*}
See \cite{MR3700208}*{\S 8.7.1}. 
Define $s=3x_0x_3-x_1x_2$. Then $Ds=f$ and $s$ is a local slice for $D$. By the Slice Theorem, $B_f=A_f[s]=A_f^{[1]}$. In particular, we find that:
\[
6f^3x_3=(x_0)s^3-(3g)s^2+(3x_0h)s-gh \quad\text{and}\quad 
s^2=h+2fF\in \mathcal{R}
\]
Therefore:
\[
x_0s^3+3x_0hs=x_0s(s^2+3h)\in R \implies s\in {\rm frac}(\mathcal{R})
\]
So ${\rm frac}(B)=({\rm frac}(A))(s)\subseteq{\rm frac}(\mathcal{R})$ implies
${\rm frac}(\mathcal{R})={\rm frac}(B)$.
\end{proof}

\begin{proposition}\label{SL2-module}
Let $W$ be a finite-dimensional $\mathrm{SL}_2$-module and let $(\delta ,\upsilon )$ be the associated linear fundamental pair on $k[W]$. 
Let $\tilde{A}=\krn\delta$ and $\epsilon =[\delta ,\upsilon]$, and define $\tilde{R}=k[\krn\delta ,\krn\upsilon ]$ and $\tilde{S}=\tilde{R}[\krn\epsilon ]$. 
\begin{itemize}
\item [{\bf (a)}] $k[W]$ is integral over $\tilde{R}$.
\item [{\bf (b)}] ${\rm frac}(\tilde{S})={\rm frac}(k[W])$.
\item [{\bf (c)}] If $\tilde{A}_1\ne\{ 0\}$ then ${\rm frac}(\tilde{R})={\rm frac}(k[W])$. 
\end{itemize}
\end{proposition}

\begin{proof} Part (a): Let $W=\bigoplus_{1\le i\le n}\V_{d_i}$ be the decomposition of $W$ into irreducible modules. Given $1\le i\le n$, let 
$k[\V_{d_i}]=k[x_{i0},\hdots ,x_{id_i}]$ with basic fundamental pair $(D_i,U_i)$ as above (i.e., $D_i(x_{ij})=x_{i,j-1}$ for $j\ge 1$ and $D_i(x_{i0})=0$, etc.). 
Let $R_i=k[\krn D_i,\krn U_i]\subset \tilde{R}$, $1\le i\le n$. 
By {\it Proposition\,\ref{free-module2}}, each $x_{ij}$ is integral over $R_i$, hence also over $\tilde{R}$. Since the $x_{ij}$ generate $k[W]$ as a $k$-algebra, 
$k[W]$ is integral over $\tilde{R}$.

Part (b). Given $1\le i\le n$, let $E_i=[D_i,U_i]$ and $S_i=R_i [\krn E_i]\subset\tilde{S}$. By {\it Proposition\,\ref{fraction-field}}, ${\rm frac}(S_i)={\rm frac}(k[\V_{d_i}])$. 
Therefore, each $x_{ij}\in{\rm frac}(\tilde{S})$, which implies ${\rm frac}(\tilde{S})={\rm frac}(k[W])$.

Part (c). If $\tilde{A}_1\ne \{0 \}$, then {\it Lemma\,\ref{useful}} implies ${\rm frac}(\tilde{R})={\rm frac}(k[W])$.
\end{proof}


\section{Fundamental Pairs for Affine Domains}\label{general-case}

Suppose that $B$ is an affine $k$-domain with nonzero fundamental pair $(D,U)$ and set $E=[D,U]$. 
Likewise, suppose that $B'$ is an affine $k$-domain with nonzero fundamental pair $(D',U')$ and set $E'=[D',U']$. 
Let $B=\bigoplus_{d\in\Z}B_d$ and $B'=\bigoplus_{d\in\Z}B_d^{\prime}$ be the $\Z$-gradings induced by the fundamental pairs (see {\it Theorem\,\ref{main1}}). 
Let $A=\krn D$ with $\N$-grading $A=\bigoplus_{d\in\N}A_d$ (see {\it Theorem\,\ref{main2}}), and likewise $A'=\krn D'=\bigoplus_{d\in\N}A_d^{\prime}$.
In this case, a $k$-algebra homomorphism $\varphi \colon B'\to B$ is {\bf equivariant} if $\varphi D'=D\varphi$ and $\varphi U'=U\varphi$. 
 
Given $f\in A_d$, let $\widehat{f}$ denote the vector space with basis 
$\{ U^jf\, |\, 0\le j\le d\}$. 
Then $\widehat{f}$ is a $(D,U)$-invariant, hence $\SL_2$-invariant, subspace; see \cite{Freudenburg2022}*{Lemma 3.5 and Lemma 3.6}. 

\begin{lemma}\label{generators} Let $f_1,\cdots ,f_r\in A$ be nonzero and homogeneous such that $A=k[f_1,\hdots ,f_r]$.
\begin{itemize}
\item [{\bf (a)}] $B=k[\widehat{f_1},\hdots ,\widehat{f_r}]$.
\item [{\bf (b)}] Let $\varphi \colon B'\to B$ be an equivariant $k$-algebra homomorphism and let 
$\tilde{\varphi}\colon A' \to A$ be its restriction. Then $\varphi$ is surjective if and only if $\tilde{\varphi}$ is surjective.
\end{itemize} 
\end{lemma} 

\begin{proof} Set $n_i=\deg f_i$. Assume that $n_1\le n_2\cdots \le n_r$ and let $t\in\N$ be such that $n_i=0$ if and only if $1\le i\le t-1$.
Let $\mathcal{S}=k[\widehat{f}_1,\hdots ,\widehat{f}_r]$. Given $n\ge 1$, let $\mathcal{F}_n=\krn D^{n+1}$. Clearly, $\mathcal{F}_0=A$ is contained in $\mathcal{S}$. 
By \cite{Freudenburg2022}*{Theorem 3.4}, the plinth ideal $I_1=A\cap DB$ equals $f_tA+\cdots +f_rA$, 
and if $h_i=U^nf_i$, $t\le i\le r$, then 
\[
\mathcal{F}_n=Ah_t+\cdots +Ah_r+\mathcal{F}_{n-1}
\]
as $A$-modules. If we assume, by way of induction, that $\mathcal{F}_{n-1}\subset\mathcal{S}$, then, since $A\subset\mathcal{S}$ and $h_i\in\mathcal{S}$ for each $i$, 
it follows that $\mathcal{F}_n\subset\mathcal{S}$. Therefore, $\mathcal{F}_n\subset\mathcal{S}$ for each $n\ge 1$. 
Since $B=\bigcup_n\mathcal{F}_n$, we conclude that $\mathcal{S}=B$. This shows part (a). 

For part (b), assume that $\tilde{\varphi}$ is surjective. Then $\varphi (A_0^{\prime})=A_0$ by equivariance.
Part (a) implies $B=A_0[\widehat{f}_t,\hdots ,\widehat{f}_r]$ for nonzero homogeneous $f_i\in A$, and there exist $F_t,\hdots ,F_r\in A'$ such that 
$\varphi (F_i)=f_i$, $t\le i\le r$. By equivariance, $\varphi (\widehat{F}_i)=\widehat{f}_i$ for each $i$. 
Therefore, $B=A_0[\widehat{f}_t,\hdots, \widehat{f}_r]$ is in the image of $\varphi$. 
Conversely, if $\varphi$ is surjective, then equivariance implies that $\tilde{\varphi}$ is surjective. 
\end{proof} 

In the lemma above, the vector spaces 
$\widehat{f_1},\hdots ,\widehat{f_r}$ constitute a 
{\bf presentation} of $B$ as an $\mathrm{SL}_2$-algebra. 
A presentation of $B$ determines an equivariant surjection of $k$-algebras
\[
\varphi :k[\V_{n_1}\oplus\cdots\oplus \V_{n_r}] \twoheadrightarrow B =k[\widehat{f_1},\dots ,\widehat{f_r}]
\quad (n_i=\deg f_i \, ,\, 1\le i\le r)
\]
such that $\varphi$ restricts to an equivariant surjection of vector spaces $\varphi (\V_{n_i})=\widehat{f_i}$ for each $i$. We call this the {\bf equivariant surjection determined by the presentation}.

Let $R=k[\krn D,\krn U]$ and $S=R[\krn E]$. Likewise, let $R'=k[\krn D',\krn U']$ and $S'=R'[\krn E']$. 

\begin{lemma}\label{equivariant} Assume that $\varphi \colon B' \to B$ is an equivariant surjection. Then
\begin{itemize}
\item [{\bf (a)}] $\varphi (B_d^{\prime})=B_d$ for each $d\in\Z$.
\item [{\bf (b)}] $\varphi (\krn D')=\krn D$ and $\varphi (\krn U')=\krn U$.
\item [{\bf (c)}] If $B'$ is integral over $R'$ (resp., $S'$) then $B$ is integral over $R$ (resp., $S$). 
\item [{\bf (d)}] If $R'\subset B'$ (resp., $S'\subset B'$) is a birational inclusion then $R\subset B$ (resp., $S\subset B$) is a birational inclusion.
\end{itemize}
\end{lemma}

\begin{proof} (a) By equivariance, $\varphi (B_d^{\prime})\subset B_d$ for each $d\in\Z$, so by surjectivity of $\varphi$ we must have $\varphi (B_d^{\prime}) = B_d$ for each $d\in\Z$.

(b) By hypothesis, $J=\krn\varphi$ is a $(D',U')$-invariant ideal of $B'$. So $(D,U)=(D'/J ,U'/J)$ is the induced fundamental pair on $B=B'/J$.
By {\it Theorem\,\ref{main3}(b)}, we have $\varphi (\krn D')=\krn (D'/J)$ and $\varphi (\krn U')=\krn (U'/J)$. 

(c) Since 
$B_0^{\prime}=\krn E'$ and $B_0=\krn E$, part (a) implies 
$\varphi (\krn E')=\krn E$. By part (b), it follows that
$\varphi (R')=R$ and $\varphi (S')=S$. 
Therefore, when $B'$ is integral over $R'$, $\varphi (B')=B$ is integral over $\varphi (R')=R$, and when $B'$ is integral over $S'$, $\varphi (B')=B$ is integral over $\varphi (S')=S$. 

(d) Assume that ${\rm frac}(B')={\rm frac}(R')$. Let $b\in B$ be given and choose $b'\in \varphi^{-1}(b)$. Write $b'=\frac{r'}{s'}$ for $r',s'\in R'$.
Then $r:=\varphi (r')\in R$ and $s:=\varphi (s')\in R$. We have:
\[
\textstyle s'b'=r' \implies sb=r \implies b=\frac{r}{s}\in {\rm frac}(R)
\]
Therefore, ${\rm frac}(B)={\rm frac}(R)$. The same argument works if $S'\subset B'$ is a birational inclusion. 
\end{proof}

The following theorem and its corollary are two of our main results.

\begin{theorem}\label{general} 
Let $B$ be an affine $k$-domain with nontrivial fundamental pair $(D,U)$ and induced $\Z$-grading $B=\bigoplus_{i\in\Z}B_i$. Define $E=[D,U]$, $A=\krn D$ and $\Omega=\krn U$. Set $R=k[A,\Omega ]$ and $S=R[\krn E]$. The following properties hold. 
\begin{itemize}
\item [{\bf (a)}] $B$ is integral over $R$. 
\item [{\bf (b)}] The inclusion $S\to B$ is birational.
\item [{\bf (c)}] If $A_1\ne\{ 0\}$ then the inclusion $R\to B$ is birational. 
\end{itemize}
\end{theorem}

\begin{proof} Part (a). Let $B$ have presentation as in {\it
Lemma\,\ref{generators}(a)}, and let $\varphi :k[W]\to B$ be the
equivariant surjection determined by this presentation, where $W$ is a
finite-dimensional $\mathrm{SL}_2$-module. 
Let $(\delta ,\upsilon)$ be the fundamental pair induced by $W$. 
By {\it Proposition\,\ref{free-module2}}, $k[W]$ is integral over $\mathcal{R}=k[\krn\delta ,\krn\upsilon ]$. By {\it Lemma\,\ref{equivariant}(b)}, $\varphi (\mathcal{R})=R$.
It thus follows by {\it Lemma\,\ref{equivariant}(c)} that $B$ is integral over $R$. 

Part (b). Let $\epsilon =[\delta ,\upsilon ]$ and $\mathcal{S}=\mathcal{R}[\krn\epsilon ]$. By {\it Proposition\,\ref{fraction-field}}, the inclusion $\mathcal{S}\to k[W]$ is birational. 
By {\it Lemma\,\ref{equivariant}(d)}, it follows that the inclusion $S\to B$ is birational. 

Part (c). If $A_1\ne \{ 0\}$, then {\it Lemma\,\ref{useful}} implies that ${\rm frac}(R)={\rm frac}(B)$. 
\end{proof}

\begin{corollary}\label{ideals} In addition to the hypotheses of {\it Theorem\,\ref{general}}, assume that $B$ is a normal affine $k$-domain. 
\begin{itemize}
\item [{\bf (a)}] $\mathcal{C}_S(B)\ne (0)$
\item [{\bf (b)}] If $A_1\ne\{ 0\}$ then $\mathcal{C}_R(B)\ne (0)$.
\end{itemize}
\end{corollary} 

\begin{proof} It is well-known that $A$ and $\Omega$ are affine; see, for example, \cite{Freudenburg2022}*{Theorem 3.4}. 
Also, $B_0=\krn E$ is affine, since it is the ring of invariants of a torus action on $B$. Therefore, $R$ and $S$ are $k$-affine. 
For the integral domain $K$, let $\mathcal{L}_K$ denote the integral closure of $K$ in its field of fractions. 

Part (a). 
Let nonzero $b\in B$ be given. By {\it Theorem\,\ref{general}(c)}, there exist nonzero $s_1,s_2\in S$ such that $b=s_1/s_2$. 
By {\it Theorem\,\ref{general}(a)}, $b=s_1/s_2$ is integral over $S$, so $b\in \mathcal{L}_S$. It follows that $B=\mathcal{L}_B=\mathcal{L}_S$.
By {\it Lemma\,\ref{conductor}}, we obtain $\mathcal{C}_S(B)\ne (0)$. 

Part (b). Assume that $A_1\ne\{ 0\}$. By {\it Theorem\,\ref{general}(b)}, there exist nonzero $r_1,r_2\in R$ such that $b=r_1/r_2$. 
By {\it Theorem\,\ref{general}(a)}, $b=r_1/r_2$ is integral over $R$, so $b\in \mathcal{L}_R$. It follows that $B=\mathcal{L}_B=\mathcal{L}_R$.
By {\it Lemma\,\ref{conductor}}, we obtain $\mathcal{C}_R(B)\ne (0)$.
\end{proof} 

\subsection{Proof of Theorem\,\ref{criterion}}
Let $\mathcal{R},\mathcal{S}\subset B$ be as defined above. 

Part(a). Assume that there exists nonzero $f\in A_2$. By {\it Corollary\,\ref{ideals}}, $\mathcal{C}_{\mathcal{S}}(B)\ne (0)$, which confirms the first condition for $(E,D,U)$ to be a compatible triple. 
Set $g=Uf$. Then $g\in B_0=\krn E$, $Dg\in A_2$ and $Ug\in\Omega_{-2}$ are all nonzero. Equivalently, 
\[
g\in (\krn D^2\setminus\krn D)\cap\krn E \quad \text{and}\quad
g\in (\krn U^2\setminus\krn U)\cap\krn E
\]
which confirms the second condition for $(E,D,U)$ to be a compatible triple.

Conversely, assume that $(E,D,U)$ is a compatible triple. Then there exists 
\[
g\in (\krn D^2\setminus\krn D)\cap\krn E
\]
and since $\krn E=B_0$, we have $g\in B_0$, so $Dg\in A_2\setminus\{ 0\}$. 

Part (b). From $A_1 \neq \{0\}$ and {\it Corollary\,\ref{ideals}} we have $\mathcal{C}_{\mathcal{R}}(B)\ne (0)$, which confirms the first condition for $(D,U)$ to be a compatible pair. By hypothesis, there exists nonzero $f\in A_1$. Then $Uf\in\Omega\setminus\{ 0\}$ and we have
\[
f\in (\krn U^2\setminus\krn U)\cap\krn D
\]
which confirms the second condition for $(D,U)$ to be a compatible pair. 

Conversely, assume that $(D,U)$ is a compatible pair. 
By\,{\it Theorem\,\ref{main2}}, if $h\in A_d$ is nonzero, then $\deg_U(h)=d$.
Therefore:
\[
(\krn U^2\setminus\krn U)\cap\krn D=A_1\setminus\{ 0\}
\]
Since $(\krn U^2\setminus\krn U)\cap\krn D\ne\emptyset$ by hypothesis, the proof is complete.
\qed

\subsection{Application to $\mathrm{SL}_2$-modules}

\begin{corollary} Let $W$ be a nontrivial finite-dimensional $\mathrm{SL}_2$-module. Let $B=k[W]$ with induced fundamental pair $(\delta ,\upsilon )$ and set $\epsilon = [\delta ,\upsilon ]$. 
Write $W=W'\oplus e\V_0$ for submodule $W'$ and maximal $e\in\N$. 
\begin{itemize}
    \item [{\bf (a)}] If $W'\ne \V_4$ then $A_2\ne\{ 0\}$ and $ (\epsilon ,\delta ,\upsilon )$ is a compatible triple for $B$.
    \item [{\bf (b)}] If $W'\ne \V_3$ and if $\V_d$ is a submodule of $W$ for some odd $d\ge 1$ then $A_1\ne\{ 0\}$ and $(\delta ,\upsilon )$ is a compatible pair for $B$.
\end{itemize}
\end{corollary} 

\begin{proof}
    Let $W'=\bigoplus_{1\le i\le n}e_iV_{d_i}$ where $d_i,e_i\in\N$, $d_i\ne d_j$ if $i\ne j$ and $d_i,e_i\ge 1$. Let $A=\krn\delta$
    with induced grading $A=\bigoplus_{d\in\N}A_d$. 

    Part (a). By {\it Theorem\,\ref{criterion}}, it suffices to show that $A_2\ne \{ 0\}$ in this case. If $\V_4\oplus \V_4$ is a submodule of $W'$, 
    let $(x_i,y_j)$ be coordinates for $\V_4\oplus \V_4$, $0\le i,j\le 3$. Then: 
    \[
    x_0y_3-x_1y_2+x_2y_1-x_3y_0\in A_2\setminus\{ 0\}
    \]
    If $\V_4\oplus \V_4$ is not a submodule of $W'$ then $d_i\ne 4$ for some $i$, and {\it Proposition\,\ref{A1-A2}} implies $A_2\ne\{ 0\}$. 

    Part (b). By {\it Theorem\,\ref{criterion}}, it suffices to show that $A_1\ne \{ 0\}$ in this case. If $\V_3\oplus \V_3$ is a submodule of $W'$, 
    let $(x_i,y_j)$ be coordinates for $\V_3\oplus \V_3$, $0\le i,j\le 2$. Then: 
    \[
    y_2(2x_0x_2-x_1^2)-y_1(3x_0x_3-x_1x_2)+y_0(3x_1x_3-2x_2^2)\in A_1\setminus\{ 0\}
    \]
    If $\V_3\oplus \V_3$ is not a submodule of $W'$ then 
    $V_d$ is a submodule for some odd $d\ne 3$ and 
    {\it Proposition\,\ref{A1-A2}} implies $A_1\ne\{ 0\}$. 
\end{proof}


\section{Applications to Smooth Complex Affine Varieties}\label{applications}

From now on we work over the ground field $\C$.

\subsection{The Algebraic Density Property}
We extend the algebraic criterion for compatible pair/triple to a criterion for the algebraic density property, aiming to find more examples of smooth affine $\mathrm{SL}_2$-varieties enjoying this property. 
First we recall Varolin's notion of algebraic density property, cf. Definition \ref{def-ADP}:

We say that a smooth complex algebraic variety $M$ has the \textbf{algebraic density property} if the Lie algebra generated by $\C$-complete algebraic vector fields on $M$ coincides with the Lie algebra of all algebraic vector fields on $M$.

\begin{theorem}[\cite{MR2385667}*{Theorem 2, Remark 2.7}, \cite{Andrist.Freudenburg.Huang.Kutzschebauch.Schott}*{Theorem 2.17}] \label{theorem: pairCriterion}
	Let $M$ be a smooth complex affine variety such that
	\begin{enumerate}
		\item[(a)]		the group $\mathrm{Aut}_{\rm alg}(M)$ of algebraic automorphisms acts transitively on $M$,
            \item[(b)]      there exist $x_0 \in M$, $v \in T_{x_0} M$ such that the image of $v$ under the induced action of the isotropy group $(\mathrm{Aut}_{\rm alg}(M))_{x_0}$ of $x_0$ on $T_{x_0} M$ spans the tangent space, and
		\item[(c)]		there exists a compatible $n$-tuple ($n \ge 2$) of complete algebraic vector fields on $M$.
	\end{enumerate}
		Then $M$ has the algebraic density property. 
\end{theorem}

For a complex variety $M$, we write $\mathrm{SAut}(M)$ for the subgroup of the algebraic automorphism group $\mathrm{Aut}_{\rm alg}(M)$, which is generated by all algebraic one-parameter unipotent subgroups of $\mathrm{Aut}_{\rm alg}(M)$ and call $\mathrm{SAut}(M)$ the \textbf{special automorphism group} of $M$. The vector field corresponding to such an algebraic one-parameter unipotent subgroup is defined by a locally nilpotent derivation.




\begin{lemma} \label{lemma: transitve-tangSemihomo}
	Let $M$ be a smooth complex affine algebraic variety of dimension at least $2$. If the special automorphism group $\mathrm{SAut}(M)$ acts transitively on $M$, then for any $x_0 \in M$, the induced action of the isotropy group $(\mathrm{Aut}_{\rm alg}(M))_{x_0}$ of $x_0$ on $T_{x_0} M\setminus \{0\}$ is transitive.
\end{lemma}
\begin{proof}
    This is a direct consequence of \cite{MR3039680}*{Theorem 4.2} if we take $G = \mathrm{SAut}(M)$. 
\end{proof}

For a smooth affine $\mathrm{SL}_2$-variety $X$, let $\C[X]= \bigoplus_{d\in\Z} \C[X]_d$ be the $\Z$-grading induced by the torus action. Denote by $D,U$ the two fundamental LNDs and with the same notation the two $\C$-complete vector fields associated with them. Furthermore for $d \in \N$ denote by $A_d = (\ker D)_d = \ker D \cap \C[X]_d$ the elements of degree $d$ in $\ker D$.

\begin{theorem} \label{criterionADP}
    Let $M$ be a smooth affine variety satisfying
    \begin{enumerate}[label=(\alph*)]
        \item     the special automorphism group $\mathrm{SAut}(M)$ acts transitively on $M$, and
        \item     there is a nontrivial algebraic $\mathrm{SL}_2$-action on $M$ with nontrivial $A_1$ or $A_2$.
    \end{enumerate}
    Then $M$ has the algebraic density property. 
\end{theorem}
\begin{proof}
    By {\it Proposition\,\ref{lemma: transitve-tangSemihomo}} transitivity of $\mathrm{SAut}(M)$-action implies the first two conditions of {\it Theorem\, \ref{theorem: pairCriterion}}. Then the algebraic density property follows from {\it Theorem\,\ref{criterion}} and {\it Theorem\, \ref{theorem: pairCriterion}}.
\end{proof}

\begin{remarks}
    (i) Arzhantsev et al.\ \cite{MR3039680} introduced the notion of flexibility for any complex variety $M$: A point $x \in M$ is called \textbf{flexible} if its tangent space $T_x M$ is spanned by finitely many locally nilpotent derivations. $M$ is called \textbf{flexible} if all regular points $M_{\mathrm{reg}}$ of $M$ are flexible. In particular, they proved the following characterization for an irreducible complex affine variety $M$ of dimension at least $2$:
    \begin{enumerate}[label=(\alph*)]
        \item The group $\mathrm{SAut}(M)$ acts transitively on $M_{\mathrm{reg}}$.
        \item The group $\mathrm{SAut}(M)$ acts infinitely transitively on $M_{\mathrm{reg}}$.
        \item $M$ is a flexible variety.
    \end{enumerate}
    Hence Condition (a) in {\it Theorem\,\ref{criterionADP}} is equivalent to $M$ being flexible.
    
    (ii) When $A_1$ is nontrivial or $\C[X]$ is a quotient of $\C[\V_3]$, then $(D, U)$ forms a compatible pair of LNDs and $M$ has the so-called algebraic overshear density property by \cite{Andrist:2024aa}*{Proposition 3.3}.
\end{remarks}

In the following, we choose flexible smooth affine $\mathrm{SL}_2$-varieties with nontrivial $A_1$ or $A_2$ in their coordinate rings, applying {\it Theorem\,\ref{criterionADP}}, to find new examples of smooth affine algebraic varieties with the algebraic density property. 

\subsection{Calogero--Moser spaces} \label{sec: CaMo}

As an example we consider the Calogero--Moser space $\mathcal{C}_n, n \in \N$ is defined as a geometric quotient:
\[
\{ (X,Y) \in \mathrm{Mat}_{n \times n}(\C) \oplus \mathrm{Mat}_{n \times n}(\C): \mathrm{rank}([X,Y]+ I_n) =1 \} /\!/ \GL_n(\C)
\]
and the $\GL_n(\C)$-action is given by simultaneous conjugation. This action is free, see Wilson \cite{MR1626461}, thus $\mathcal{C}_n$ is a smooth affine variety of dimension $2n$. 
We also have a nontrivial $\mathrm{SL}_2$-action 
\[
    \begin{pmatrix}
        a & b \\ c & d
    \end{pmatrix}  \cdot (X, Y) = (a X + b Y, c X + d Y)
\]
which preserves the commutator $[X, Y]$, and induces the fundamental pair 
\[
    D = X \frac{\partial}{\partial Y} = \sum_{i,j=1}^n X_{ij} \frac{\partial}{\partial Y_{ij}}, \quad U = Y \frac{\partial}{\partial X} = \sum_{i,j=1}^n Y_{ij} \frac{\partial}{\partial X_{ij}}
\]    
The group generated by the special automorphisms 
\begin{align*}
    (X, Y) \mapsto (X, Y + t X^k) \\
    (X, Y) \mapsto (X + t Y^k, Y)
\end{align*}
$k \in \N_0, t \in \C$, acts transitively on $\mathcal{C}_n$, see Berest and Wilson \cite{MR1785579}. Moreover, the function $f=\tr X$ is in $\ker D$ and $\deg_U f = 1$. Hence by {\it Theorem\,\ref{criterionADP}} the Calogero--Moser space $\mathcal{C}_n$ has the algebraic density property.

\subsection{Quiver varieties}
Before introducing the next example, we first recall the definition of a quiver variety.

Take a quiver $Q = (I, E)$, where $I$ is the vertex set and $E$ the edge set. For an edge $a \in E$ with $a \colon i \to j$, denote the head $h(a) = j$ and the tail $t(a) = i$. The double $\Bar{Q}$ of $Q$ is the quiver obtained by adding an arrow $a^\ast \colon  j \to i$ for each arrow $a \colon i \to j$ in $Q$. The space of representations of $Q$ with dimension vector $\beta \in \N^I$ is 
\begin{align*}
    \mathrm{Rep}(Q, \beta) = \bigoplus_{a \in E} \mathrm{Mat}_{\beta_{h(a)} \times \beta_{t(a)}}( \C )
\end{align*}
Each component $\beta_i$ of $\beta$ is the dimension of the vector space $V_i \cong \C^{\beta_i}$ at the vertex $i$. Thus representations in the same isomorphism class belong to the same $G(\beta)$-orbit, where 
\[
    G(\beta) = \left( \prod_{i \in I} \GL_{\beta_i}(\C) \right) / \C^\ast
\]
acts on $\mathrm{Rep}(Q, \beta)$ by conjugation. To make the action effective, quotient out one $\C^*$ for every loop in the quiver. 

For the double $\Bar{Q}$ we can identify $\mathrm{Rep}(\Bar{Q}, \beta)$ with $\mathrm{T}^\ast \mathrm{Rep}(Q, \beta)$ using the trace pairing. Hence $\mathrm{Rep}(\Bar{Q}, \beta)$ can be endowed with the $G(\beta)$-invariant symplectic form $\omega$ given by $\sum_{a \in E} \tr (d X_a \wedge d X_{a*})$. Then the $G(\beta)$-action induces the moment map 
\begin{align*}
    \mu_\beta &\colon \mathrm{Rep}(\bar{Q}, \beta) \to \mathrm{End}(\beta)_0, \\ 
    \mu_\beta(X)_i &= \sum_{a \in E, h(a) = i} X_a X_{a^*} - \sum_{a \in E, t(a) = i} X_{a^*} X_a
\end{align*}
where $\mathrm{End}(\beta)_0 = \{ (\theta_i) : \sum_{i \in I} \tr (\theta_i) = 0 \}$ is contained in $\mathrm{End}(\beta) = \bigoplus_{i \in I} \mathrm{Mat}_{\beta_i \times \beta_i}(\C)$. For $\eta \in \C^I$, let $\theta^\eta = (\theta^\eta_i) = (\eta_i I_{\beta_i}) \in \mathrm{End}(\beta)$. If $\eta \cdot \beta = 0$ then $\theta^\eta \in \mathrm{End}(\beta)_0 $. 

Given $\eta \in \C^I$ and $\beta \in \N^I$ with $\eta \cdot \beta = 0$, the quiver variety associated to $Q$ is the GIT quotient
\[
    \mathfrak{M}^\eta (Q, \beta) = \mu^{-1}_\beta(\theta^\eta) /\!/ G(\beta)
\]
which is a reduced irreducible scheme that describes the isomorphism classes of semisimple representations of the so-called deformed preprojective algebra with dimension vector $\beta$, see Crawley-Boevey \cite{MR1834739}. 

\subsection{Calogero--Moser spaces with rank (at most) two condition}

One natural generalization of the example in {\it Section\,\ref{sec: CaMo}} is to allow also higher ranks. For rank at most two the following complex affine variety $N_{n, \tau}, n \in \N$ was studied by Bielawski and Pidstrygach in \cite{MR2739794}
\begin{align*}
	 \{ (A, B, v, w) \in \mathrm{T}^\ast (\matnn \oplus \matv): [A, B] - vw = \tau \cdot I_n \} /\!/ \GL_n(\C)
\end{align*}
where $\tau \in \C$, and the $\GL_n(\C)$-action is given by 
\begin{align} \label{equation: GLn-action}
		g \cdot (A, B, v, w) = (g A g^{-1}, g B g^{-1}, g v, w g^{-1} )
\end{align}
where $A, B \in \matnn, v \in \matv, w \in \matw$ and $g \in \GL_n(\C)$. Here we use the natural identification $( \matv )^\ast \cong \matw $ and write $\mathfrak{gl}_k$ for the Lie algebra of the Lie group $\GL_k (\CC)$. For such a generalized Calogero--Moser space, Gibbons and Hermsen \cite{MR0761664} suggested the physical interpretation with particles that possess $r$ inner degrees of freedom, while we consider the case of rank $r$ at most two. Note that $\mathcal{C}_n$ can be embedded into $N_{n,1}$ by choosing e.g. $v = ( v_1, 0)$, $w = \begin{pmatrix} w_1 \\ 0 \end{pmatrix}$ and that for any $\tau, \tau' \in \C^\ast$, $\ins$ and $N_{n, \tau'}$ are isomorphic to each other. 

To interpret the affine variety $\ins$ as a symplectic quotient of $\mathrm{T}^\ast (\matnn \oplus \matv)$ by $\GL_n(\CC)$, we start with the following moment map $\mu$ for the $\GL_n(\CC)$-action in \eqref{equation: GLn-action}
\begin{align*}
	\mu \colon \mathrm{T}^\ast(\matnn \oplus \matv) \to \mathfrak{gl}_{n}, (A, B, v, w) \mapsto [A, B] - vw
\end{align*}
When $\tau \neq 0$, $\GL_n(\CC)$ acts freely on $\mu^{-1}(\tau \cdot I_n)$ and hence 
\[
	\ins = \mu^{-1}(\tau \cdot I_n) / \GL_n(\CC)
\]
is smooth. 

In \cite{MR2739794} $\ins$ was identified as a subvariety in the
moduli space of representations in $(\CC^n, \CC^1)$ of the double of
the following quiver $Q$:
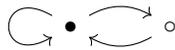
\[
\begin{tikzcd}
\bullet \arrow[out=150,in=210,loop] \arrow[r,bend left]  & \circ \arrow[l,bend left,swap]
\end{tikzcd}
\]

To make this embedding explicit, one extends
the pair $(A,B)$ into a pair of square matrices $(\hat{A}, \hat{B})$
of size $n+1$, and then fills the last column and the last row with the components of $v$ and $w$ such that the original $\GL_n(\CC)$-action is induced by an appropriate action by conjugation on  $(\hat{A}, \hat{B})$. More precisely: Write $v = ( v_1, v_2)$, $w = \begin{pmatrix} w_1 \\ w_2 \end{pmatrix}$ and let
\begin{align} \label{equation: hatrep}
	\hat{A} = \begin{pmatrix} A & v_1 \\ w_2 & 0 \end{pmatrix} 
	\quad
	\hat{B} = \begin{pmatrix} B & v_2 \\ -w_1 & 0 \end{pmatrix}
\end{align}
Now identify the group $\GL_n(\CC)$ with the subgroup $H$ of $\GL_{n+1}(\CC)$ which consists of 
\begin{align*}
	\begin{pmatrix}	L & 0 \\ 0 & 1 \end{pmatrix}, \quad L \in \GL_n(\CC) 
\end{align*}
The $\GL_n(\CC)$-action on the quadruple $(A, B, v, w)$ turns into an $H$-conjugation on the pair $(\hat{A}, \hat{B})$, which in turn gives the new moment map $\mu_H$ as the upper-left $(n \times n)$-minor of the commutator 
\[
	[\hat{A}, \hat{B}] = \begin{pmatrix}  [A, B] - vw & A v_2 - B v_1 \\ w_2 B + w_1 A & \ast \end{pmatrix}
\]

Let $\bar{Q}$ be the double of the quiver $Q$ and $\mathrm{Rep}(\bar{Q}, \beta) \cong \mathrm{T}^\ast \mathrm{Rep} (Q, \beta)$ be the symplectic vector space of all representations of $\bar{Q}$ with dimension vector $\beta = (n,1)$. An element of $\mathrm{Rep} (\bar{Q}, \beta)$ can be written as
\[	( A, B, X_1, X_2, Y_1, Y_2)	\]
where $A, B \in \matnn, X_1, X_2 \in \mathrm{Mat}_{n \times 1}(\C), Y_1, Y_2 \in \mathrm{Mat}_{1 \times n}(\C)$. The following action
\begin{align*}
	&(g, h) \cdot (A, B, X_1, X_2, Y_1, Y_2) = \\  &( g A g^{-1}, g B g^{-1}, g X_1 h^{-1}, g X_2 h^{-1}, h Y_1 g^{-1}, h Y_2 g^{-1} ), 
\end{align*}
of $G (\beta) = \left( \GL_n(\CC) \times \GL_1(\CC) \right) / \CC^\ast$ on $\mathrm{Rep}(\bar{Q}, \beta)$ with $g \in \GL_n(\CC), h \in \GL_1 (\CC)$ is effective and induces a moment map $\mu_\beta$ that takes $(A, B, X_1, X_2, Y_1, Y_2)$ to
\[
	\left( [A, B] + X_1 Y_2 - X_2 Y_1, Y_1 X_2 - Y_2 X_1 \right) \in \mathfrak{gl}_n \oplus \mathfrak{gl}_1 
\]
see \cite{MR2739794}*{Part 2}.
For the fixed point $O$ of coadjoint $G(\beta)$-action
\[	O = \begin{pmatrix} \tau \cdot I_{n} & 0 \\ 0 & - n \tau \end{pmatrix} \in \mathfrak{gl}_n \oplus \mathfrak{gl}_1
\]
the symplectic quotient $\mu_\beta^{-1}(O) / G(\beta)$ is per definition the quiver variety $\mathfrak{M}^\eta(Q, \beta)$ associated to $Q$ with $\eta = (1, -n)$. Moreover, it is isomorphic to the affine variety $\ins$ by the mapping
\begin{align*}
	X_1 \mapsto -v_1, \, X_2 \mapsto v_2, \, Y_1 \mapsto w_1, \, Y_2 \mapsto w_2
\end{align*}

Denote by $\CC \bar{Q}$ the path algebra of $\bar{Q}$. This is an algebra over the ring $R := \CC^2$ with component-wise multiplication, generated by the arrows in $\bar{Q}$ with the concatenation as multiplication. The idempotents of $R$ correspond to constant paths at the vertices of $Q$. The group $\mathrm{Aut}_R (\CC \bar{Q})$ of $R$-automorphisms of the path algebra $\CC \bar{Q}$ acts on the symplectic vector space $\mathrm{Rep} (\bar{Q}, \beta)$, see \cite{MR2739794}*{Section 5}. The stabilizer 
\[	\mathrm{Aut}_R (\CC \bar{Q}; c) =  \big \{ \phi \in \mathrm{Aut}_R (\CC \bar{Q}) : \phi (c) = c \big \}
\]
acts on $\mathrm{Rep}(\bar{Q}, \beta)$ such that the moment map $\nu$ for the action of $G (\beta)$ is preserved, see \cite{MR2739794}*{Section 6}.
Here,
\[
	c := [ a, a^\ast] + [x, x^\ast] + [y, y^\ast], \quad Q:
	\begin{tikzcd}
	\bullet \arrow[out=150,in=210,loop, "a"] \arrow[r,bend left, "y"]  & \circ \arrow[l,bend left,swap, "x"]
	\end{tikzcd}
\]
with $a^\ast, x^\ast, y^\ast$ denoting the remaining arrows of
$\bar{Q}$, which are the reverse arrows of $a, x, y$, respectively,
and $[a, a^\ast]$ denoting the commutator $a a^\ast - a^\ast a$.
Hence, the subgroup $\mathrm{Aut}_R (\CC \bar{Q}; c)$ acts on the
quiver variety $\mathfrak{M}^\eta(Q, \beta)$, which is isomorphic to
the variety $\ins$. Thus, we can identify $\mathrm{Aut}_R (\CC \bar{Q}; c)$ as a subgroup of $\mathrm{Aut} ( N_{n, \tau} )$.

\subsubsection{Transitive $\mathrm{SAut}$-action}
Let us consider the subgroup 
\[  T := \mathrm{TAut}_R (\C \bar{Q}; c)
\]
of $\mathrm{Aut}_R (\C \bar{Q}; c)$ generated by strictly triangular automorphisms and $\mathrm{Aff}_c$. An automorphism in $\mathrm{Aut}_R (\C \bar{Q}; c)$ is strictly triangular if it is identity on $\C Q$. The subgroup $\mathrm{Aff}_c$ consists of the affine transformations of $\mathrm{span}(a, a^*, x, y, x^*, y^*)$ preserving $c$ and equals $\mathrm{ASL}_2(\C) \times \mathrm{GL}_2(\C)$. Here, $\mathrm{ASL}_2(\C)$ is the group of affine transformations on $\C_{a, a^*}^2$ and $\mathrm{GL}_2(\C)$ acts on $\mathrm{span}(x, y, x^*, y^*)$ as 
\[
    \begin{pmatrix} -x \\ y^* \end{pmatrix}
    \mapsto
    g  \begin{pmatrix} -x \\ y^* \end{pmatrix}, \quad 
    \begin{pmatrix} (x^* & y) \end{pmatrix} \mapsto \begin{pmatrix} (x^* & y) \end{pmatrix} g^{-1} \text{ for } g \in \mathrm{GL}_2(\C).
\]

In \cite{MR2739794}*{Theorem 8.1} it is shown that $T$ acts transitively on $N_{n, \tau}$ when $\tau \neq 0$. 
In their subsequent discussion \cite{MR2739794}*{Remark 8.3} they also pointed out that the center of $T$ acts trivially on $\ins$. 
Since all other generators of the group $T$ are special automorphisms, cf.\  \cite{MR2739794}*{Section 7}, it follows that the subgroup $\mathrm{SAut} (\ins)$ acts transitively on $\ins$. Therefore we have the following

\begin{lemma} \label{lemma: transitivity}
	$\mathrm{SAut} (\ins)$ acts transitively on $\ins$ when $\tau \neq 0$.
\end{lemma}

\subsubsection{The \texorpdfstring{$\mathrm{SL}_2$}--action}
Consider the following $\mathrm{SL}_2(\mathbb{C})$-action $\varphi \colon \mathrm{SL}_2(\mathbb{C}) \to \mathrm{End}(\mathrm{T}^\ast \mathfrak{s}_{n+1})$ 
\begin{align} \label{Ins-SL2}
	\varphi \left( \begin{pmatrix} a & b \\ c & d \end{pmatrix} \right) ( \hat{A}, \hat{B}) = ( a \hat{A} + b \hat{B}, c \hat{A} + d \hat{B} ) 
\end{align}

\begin{lemma} \label{lemma: sl2-action}
	The action $\varphi$ induces an $\mathrm{SL}_2(\mathbb{C})$-action on the quotient $N_{n, \tau}$. 
\end{lemma}
\begin{proof}	
It is clear that $\varphi$ commutes with the $\GL_n(\C)$-action. 
To see that it preserves the condition $[A, B] - vw = \tau \cdot I_n$, we translate this action back to the representation $(A, B, v, w)$. Spelling out the action for each component we get
\begin{align*}
	a \hat{A} + b \hat{B} &= 
	\begin{pmatrix}	aA + b B & a v_1 + b v_2 \\ a w_2 - b w_1 & 0 \end{pmatrix}
	\\
	c \hat{A} + d \hat{B} &= 
	\begin{pmatrix}	cA + d B & c v_1 + d v_2 \\ c w_2 - d w_1 & 0 \end{pmatrix}
\end{align*}
A comparison with the entries in \eqref{equation: hatrep} then shows that $\mathrm{SL}_2(\mathbb{C})$ acts as
\begin{align*}
	(A, B) &\to ( a A + b B, c A + d B) \\
	(v_1, v_2) &\to (a v_1 + b v_2, c v_1 + d v_2 ) \\
	(w_1, w_2) &\to ( d w_1 - c w_2, - b w_1 + a w_2) 
\end{align*}
Therefore
\begin{align*}
	&[A, B] - vw \\
	\mapsto & (ad - bc) [A, B] - \left( \begin{pmatrix} a & b \\ c & d \end{pmatrix} \cdot \begin{pmatrix} v_1 \\ v_2 \end{pmatrix} \right)^T
	\begin{pmatrix} d & -c \\ -b & a \end{pmatrix} \cdot \begin{pmatrix} w_1 \\ w_2 \end{pmatrix}   
\end{align*}
which equals $[A, B] - v w $ by the defining relation of $\mathrm{SL}_2(\C)$. 
\end{proof}

\subsubsection{The function}
The fundamental pair induced by \eqref{Ins-SL2} is 
\begin{align*}
    D &= \hat{A} \frac{\partial}{\partial \hat{B}} = A \frac{\partial}{\partial B} + v_1 \frac{\partial}{\partial v_2} - w_2 \frac{\partial}{\partial w_1} \\
    U &= \hat{B} \frac{\partial}{\partial \hat{A}} = B \frac{\partial}{\partial A} + v_2 \frac{\partial}{\partial v_1} - w_1 \frac{\partial}{\partial w_2}
\end{align*}
Hence the function $f = \tr \hat{A} = \tr A$ is in $\ker D$ and $\deg_U f = 1$. By {\it Theorem\,\ref{criterionADP}} the Calogero--Moser space $N_{n, \tau}, \tau \neq 0$ with rank (at most) two condition has the algebraic density property.

\subsection{Cyclic quiver varieties}
Another natural generalization of $\mathcal{C}_n$ is to allow
$m$-tuples of matrices, and was studied in Chen--Eshmatov--Eshmatov--Tikaradze \cite{MR4418718}.

\subsubsection{Cyclic quiver}
For $m \ge 1$, let $Q_m = ( I_\infty, E_{\infty} )$ be the quiver obtained from the $m$-th extended Dynkin quiver of type $A$ by adding one more vertex and an arrow $a_\infty \colon \infty \to 0$. Let $a_i \colon i+1 \to i$ be the arrows of the extended Dynkin quiver. We have now the vertex set $I_\infty = \{ 0,1,\dots, m-1, \infty \}$ and the edge set $E_\infty = \{ a_0, \dots, a_{m-1}, a_\infty \}$. 
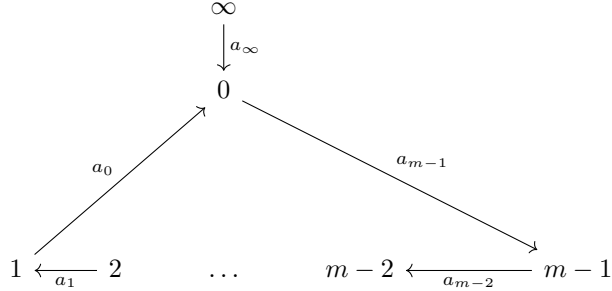
\[
\begin{tikzcd}
     & & \infty  \arrow[d, "a_\infty"] & & \\
     & &  0   \arrow[dddrrr, "a_{m-1}"] & & \\
     & & & & \\
     & & & & \\
     1 \arrow[uuurr, "a_0"] & 2 \arrow[l, "a_1"] & \dots & m-2 & & m-1 \arrow[ll, "a_{m-2}"]
\end{tikzcd}
\]
For $\alpha \in \N^m,  \lambda = (\lambda_i) \in \C^m$, consider the quiver variety $\mathfrak{M}^\eta (Q_m, \beta)$, called the \textbf{cyclic quiver variety}, where $\eta = (- \lambda \cdot \alpha, \lambda)$ and $\beta = (1, \alpha)$. By the choice of parameters we have $\eta \cdot \beta^t = 0$. A point in $\mathfrak{M}^\eta (Q_m, \beta)$ can be presented by the following representation of the double quiver $\bar{Q}_\infty$: 
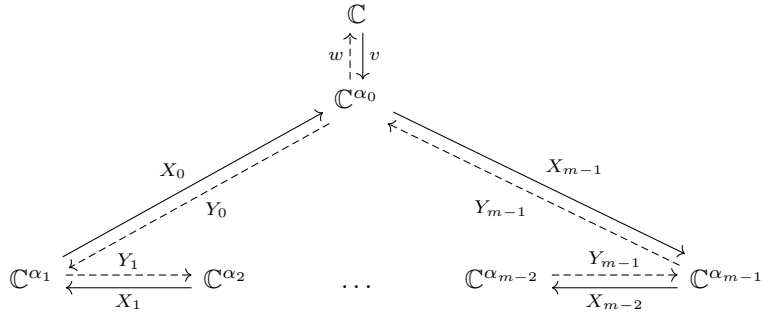
\[
\begin{tikzcd}
     & & & \C  \arrow[d, "v", shift left] & & & \\
     & & &  \C^{\alpha_0}  \arrow[u, "w", shift left, dashed]  \arrow[dddrrr, "X_{m-1}", shift left]  \arrow[dddlll, "Y_0", shift left, dashed] & & & \\
     & & & & & & \\
     & & & & & & \\
     \C^{\alpha_1} \arrow[uuurrr, "X_0", shift left] \arrow[rr, "Y_1", shift left, dashed] & & \C^{\alpha_2} \arrow[ll, "X_1", shift left] & \dots & \C^{\alpha_{m-2}} \arrow[rr, "Y_{m-1}", shift left, dashed] & & \C^{\alpha_{m-1}} \arrow[ll, "X_{m-2}", shift left] \arrow[uuulll, "Y_{m-1}", shift left, dashed]
\end{tikzcd}
\]
More precisely, a point in the cyclic quiver variety is represented by a $(2m+2)$-tuple $(X_i, Y_i, v, w)$, where $X_i = X_{\alpha_i}, Y_i = X_{\alpha^*_i}$ for $i = 0, 1, \dots, m-1$ and $v = X_{a_\infty}, w = X_{a^*_\infty}$, that satisfy the relations
\begin{align*}
    X_0 Y_0 - Y_{m-1} X_{m-1} + v w &= \lambda_0 I_{\alpha_0} \\
    X_i Y_i - Y_{i-1} X_{i-1} &= \lambda_i I_{\alpha_i}, \quad i =1, \dots, m-1 \\
    w v &=  \lambda \cdot \alpha
\end{align*}
The reductive group $
    G(\beta) =  ( \prod_{i =0}^{m-1} \GL_{\alpha_i}(\C)  ) / \C^\ast$
    acts by
\begin{align*}
    (g_0, \dots, g_{m-1}) . (X_i, Y_i, v, w) = ( g_i X_i g_{i+1}^{-1}, g_{i+1}Y_i g_i^{-1}, g_0 v, w g_0^{-1} )
\end{align*}
A convenient shorthand is to use block matrices 
\begin{align} \label{bigmatrix}
    X = 
    \begin{pmatrix}
        0 & X_0 & 0 &  \dots & 0 \\
        0 & 0 & X_1 &  \dots & 0 \\
        0 & 0 & 0  &  \ddots & \vdots \\
        \vdots & \vdots & \ddots &  \ddots & X_{m-2} \\
        X_{m-1} & 0 & \dots & 0 & 0 
    \end{pmatrix}, \quad
    Y = 
    \begin{pmatrix}
        0 & 0 & 0 &  \dots & Y_{m-1} \\
        Y_0 & 0 & 0 &  \dots & 0 \\
        0 & Y_1 & 0  &  \ddots & \vdots \\
        \vdots & \vdots & \ddots &  \ddots & 0 \\
        0 & 0 & \dots & Y_{m-2} & 0 
    \end{pmatrix}
\end{align}
and set
\begin{align*}
    \mathbf{v} &= ( v^t, 0_{\alpha_1}, \dots, 0_{\alpha_{m-1}})^t, \,
    \mathbf{w}= ( w, 0_{\alpha_1}, \dots, 0_{\alpha_{m-1}}) \\
    \Lambda &= \mathrm{Diag} (\lambda_0 I_{\alpha_0}, \dots, \lambda_{m-1} I_{\alpha_{m-1}} )
\end{align*}
Then the previous relations can be more neatly expressed as $XY - YX + \mathbf{v} \mathbf{w}= \Lambda$ and $\mathbf{w}\mathbf{v}=- \lambda \cdot \alpha$. Hence a quadruple $(X, Y, \mathbf{v}, \mathbf{w})$ satisfying these relations presents a point in the cyclic quiver variety $\mathfrak{M}^\eta(Q_m, \beta)$. 

From the choice $\eta = (- \lambda \cdot \alpha, \lambda)$ and $\beta = (1, \alpha)$ we have $\eta \cdot \beta^t = 0$. In this case the dimension of the cyclic quiver variety is given by $2 \alpha_0 - \sum_{i = 0}^{m-1}(\alpha_i -\alpha_{i+1})^2$, see \cite{MR4418718}*{Lemma 7}. 

\begin{example}
    For $\alpha = (n, \dots, n) \in \N^m$, then we have 
    \[
        \beta = (1, n, \dots, n), \quad \eta = ( - n \sum_{i=0}^{m-1} \lambda_i, \lambda_0, \dots, \lambda_{m-1} )
    \]
    and the cyclic quiver variety $\mathfrak{M}^\eta(Q_m,\beta)$ is called the $n$-th Calogero--Moser space $\mathcal{C}^\eta_n(Q_m)$ associated with $Q_m$, which is a symplectic smooth affine variety of dimension $2n$, see Etingof--Ginzburg \cite{MR1881922}. In particular, when $m=1$ this variety is simply the Calogero--Moser space $\mathcal{C}_n$ in {\it Section\,\ref{sec: CaMo}}. 
\end{example}

\begin{remark}
    Given $\eta \in \C^I$ with $\eta_i \neq 0$ and $\beta = (1, \alpha)$ with $\eta \cdot \beta^t = 0$, the cyclic quiver variety $\mathfrak{M}^\eta(Q_m, \beta)$ associated to $Q_m$ is actually isomorphic to the Calogero--Moser space $\mathcal{C}^{w(\eta)}_n(Q_m)$ associated to $Q_m$, where $w \in W_\infty$ satisfying $w(\beta) = (1, n, \dots, n)$ and $2n$ is the dimension of $\mathfrak{M}^\eta(Q_m, \beta)$. Here $W_\infty$ is the subgroup of the Weyl group of $Q_m$ fixing $\infty$. For more details see \cite{MR1834739}*{Corollary 14}.  
\end{remark}

\subsubsection{Transitive $\mathrm{SAut}$-action}
Let $(X, Y, \mathbf{v}, \mathbf{w})$ represent a point in $ \mathfrak{M}^\eta(Q_m,\beta)$. The subgroup $G$ of $\mathrm{SAut}( \mathfrak{M}^\eta(Q_m,\beta))$ generated by 
\begin{align*}
    (X, Y, \mathbf{v}, \mathbf{w}) &\mapsto (X+ t Y^{km-1}, Y, \mathbf{v}, \mathbf{w}) \\
    (X, Y, \mathbf{v}, \mathbf{w}) &\mapsto (X, Y + t X^{km-1}, \mathbf{v}, \mathbf{w})
\end{align*}
$t \in C, k \in \N$, acts transitively on $ \mathfrak{M}^\eta(Q_m,\beta)$, see \cite{MR1834739}*{Theorem 2}.

\subsubsection{The $\mathrm{SL}_2$-action}
For $m = 2$ the action is the following:
\begin{align} \label{SL2-Qm}
    \begin{pmatrix}
        a & b \\ c & d
    \end{pmatrix}
    \cdot
    (X, Y, \mathbf{v}, \mathbf{w}) = (aX + b Y, c X + dY, \mathbf{v}, \mathbf{w})
\end{align}
under which the relations $[X,Y] + \mathbf{v} \mathbf{w}= \Lambda$ and $\mathbf{w}\mathbf{v}=- \lambda \cdot \alpha$ are invariant thanks to the determinant condition $ad-bc=1$.

\subsubsection{The function}
It remains to find a function in $A_1$ or $A_2$. The fundamental pair induced by the $\mathrm{SL}_2$-action \eqref{SL2-Qm} is 
\begin{align*}
    D = X \frac{\partial}{\partial Y}, \quad
    U = Y \frac{\partial}{\partial X}
\end{align*}
Let us take the function 
\begin{align*}
    f = \tr XY = \tr (X_0 Y_0) + \tr (X_1 Y_1) 
\end{align*} 
Then  
\begin{align*}
    D f = \tr X^2 = 2 \tr (X_0 X_1 ), \quad Uf = \tr Y^2 =  \tr (Y_{1} Y_0)
\end{align*}
and $D(D(f)) = U(U(f)) = 0$. Thus $Df$ is in $\ker D$ and $\deg_U f = 2$, namely $D f \in A_2 = (\ker D)_2$. By {\it Theorem\,\ref{criterionADP}} the smooth cyclic quiver variety $\mathfrak{M}^\eta(Q_2,\beta)$ enjoys the algebraic density property.
\medskip

For $m \ge 3$ we do not know if there is an $\SL_2$-action on $\mathfrak{M}^\eta(Q_m,\beta)$. Since its algebraic automorphism group is large enough to admit transitive action, it is natural to ask the following.

\begin{question}
    Does the smooth cyclic quiver variety $\mathfrak{M}^\eta(Q_m,\beta)$ with $m \ge 3$ have the algebraic density property?
\end{question}


\appendix 
\section{Proofs for {\it Section \ref{fund-pair}}} \label{proof-fund-pair}

\subsection{Proofs for Theorem\,\ref{main1} and Theorem\,\ref{main2}} The proofs for these theorems consist in showing the following sequence of lemmas.

\cite{Freudenburg2022} gives the identity:
\begin{equation}\label{identity}
D^mU^n=D^{m-1}U^{n-1}(UD+nE-n(n-1)I) \quad ,\,\, m,n\ge 1
\end{equation}
By induction, we obtain
\begin{equation}\label{identity2}
D^nU^nf=n!\,p_n(E)f \quad {\rm for\,\, all}\quad n\ge 1 \,\, {\rm and}\,\, f\in A
\end{equation}
where the polynomials $p_n(x)$ are defined as follows.
\begin{definition} {\rm Given the integer $n\ge 1$, define $p_n(x),q_n(x)\in\Z [x]$ by:}
\[
p_n(x)=\prod_{1\le i\le n}(x-(i-1)) \quad {\rm and}\quad q_n(x)=\prod_{1\le i\le n}(x+(i-1))
\]
\end{definition}
We have further relations
\[
[D^mU^n,E]=2(n-m)D^mU^n  \quad {\rm and}\quad [U^nD^m,E]=2(n-m)U^nD^m \quad {\rm for\,\, all} \quad m,n\ge 0
\]
which follow by induction using equation \eqref{identity} and the base cases 
\[
    [D^m, E] = -2m D^m, \quad [U^n, E] = 2n U^n, \quad [DU, E] = 0
\]
where the last identity follows by evaluating on homogeneous elements. In particular, $[D^nU^n,E]=[U^nD^n,E]=0$ for all $n\ge 0$. 

Let $\mathcal{S},\mathcal{T}\subset {\rm End}_k(B)$ be the commutative subalgebras generated by $\{ DU,E\}$ and $\{ UD,E\}$, respectively.
Let $B'\subseteq B$ be the subalgebra generated by $\{\krn (E-dI)\, |\,d\in \Z\}$.

\begin{lemma}\label{subalgebra} For all $n\in\N$, $U^nD^n\in\mathcal{S}$ and $D^nU^n\in\mathcal{T}$.
\end{lemma}

\begin{proof} For $n=1$ we have $DU=E+UD\in\mathcal{T}$. For $n\ge 2$ assume that $D^{n-1}U^{n-1}\in\mathcal{T}$. Then by equation (\ref{identity}) we see that 
\[
D^nU^n=D^{n-1}U^{n-1}(UD+nE-n(n-1)I)\in\mathcal{T}
\]
It follows by induction that $D^nU^n\in\mathcal{T}$ for all $n\ge 0$. The corresponding statement for $\mathcal{S}$ follows by symmetry.
\end{proof}

\begin{lemma}\label{sum-decomp} Given the integer $n\ge 1$ and distinct $a_1,\hdots ,a_n\in\Z$,
\[
\krn\prod_{i=1}^n(E-a_iI)=\bigoplus_{i=1}^n\krn (E-a_iI)
\]
as $B_0$-modules. Moreover, $B'=\bigoplus_{ d \in\Z}\krn (E- d I)$ is a $\Z$-grading of $B'$. 
\end{lemma}

\begin{proof} We proceed by induction on $n$, the case $n=1$ being clear. 
For $n\ge 2$, consider the sequence
\[
0\to \krn (E-a_nI)\xrightarrow{i}\krn\prod_{i=0}^n(E-a_iI)\xrightarrow{\pi}\krn\prod_{i=0}^{n-1}(E-a_iI)\to 0
\]
where $i$ is the inclusion and $\pi (f)=(E-a_nI)f$. Let 
\[
    j \colon \krn\prod_{i=0}^{n-1}(E-a_iI)\to\krn\prod_{i=0}^n(E-a_iI)
\]
be the inclusion. 
By the inductive hypothesis, we have the equality of $A_0$-modules:
\[
\krn\prod_{i=1}^{n-1}(E-a_iI)=\bigoplus_{i=1}^{n-1}\krn (E-a_iI)
\]
Given $f\in \krn (E-a_if)$ for $i<n$ we have:
\[
 \pi (f)=(E-a_nI)f=(E-a_iI)f-(a_n-a_i)f=-(a_n-a_i)f
 \]
 Therefore, $\pi$ is surjective, and the sequence is exact. 
Suppose that $f\in \krn (E-a_iI)\cap\krn\pi$ for some $i<n$. Then:
\[
 0=\pi (f)=(E-a_nI)f=(E-a_iI)f-(a_n-a_i)f=-(a_n-a_i)f \implies f=0
 \]
Therefore, $\pi j$ is bijective and $j$ is a section, so the sequence splits. 
The fact that
\[
B'=\bigoplus_{d\in\Z}\krn (E-dI)
\]
as $B_0$-modules now follows by induction.

Given $d,e\in\Z$ choose $f\in\krn (E-dI)$ and $g\in\krn (E-eI)$. Then
\[
E(fg)=fEg+gEf=f(eg)+g(df)=(d+e)(fg)
\]
which implies $fg\in\krn (E-(d+e)I)$. Therefore, the given decomposition is a $\Z$-grading of $B'$.
\end{proof} 

\begin{lemma}\label{min-poly} Given nonzero $f\in A$, $p_{d+1}(E)f=0$ where $d=\deg_Uf$. Given nonzero $h\in\Omega$, $q_{e+1}(E)h=0$ where $e=\deg_Dh$. 
\end{lemma}

\begin{proof} 
By equation (\ref{identity}), we have:
\[
D^{d+1}U^{d+1}=D^dU^d(UD+(d+1)E-(d+1)dI)
\]
Applying this to $f$ gives $0=D^{d+1}U^{d+1}f=(d+1)D^dU^d(E-dI)f$, and by repeated application, $0=D^{d+1}U^{d+1}f=(d+1)!\,p_{d+1}(E)f$. The corresponding statement for $\Omega$ holds by symmetry of $D$ with $U$. 
\end{proof}

\begin{lemma}\label{ker-image} $\Omega\cap DB=\{ 0\}$
\end{lemma}

\begin{proof} 
Let $h\in \Omega\cap DB$ be given. Assume that $h\ne 0$. Choose $f\in B$ with $h=Df$, noting that $f\ne 0$. 
By {\it Lemma\,\ref{min-poly}}, we have $q_e(E)h=0$. We claim that, for all $n\ge 1$:
\[
D^nU^nf=n!\,p_n(E)f
\]
The case $n=1$ follows from:
\[
 DUf=(E+UD)f=Ef+Uh=Ef=p_1(E)f
 \]
 Assume $n\ge 2$ and $D^{n-1}U^{n-1}f=(n-1)!\,p_{n-1}(E)f$. By equation (\ref{identity}) we have
 \[
 D^nU^nf=D^{n-1}U^{n-1}(UD+nE-n(n-1)I)f=n D^{n-1}U^{n-1} (E-(n-1)I)f
 \]
 and since $[D^{n-1}U^{n-1},E]=0$ we obtain:
\[
D^nU^nf=n(E-(n-1)I) D^{n-1}U^{n-1} f=n(E-(n-1)I)(n-1)!\,p_{n-1}(E)f=n!\,p_n(E)f
\]
So the claim is proved by induction on $n$. Choose $n\ge 1$ so that $U^nf=0$. Then:
\[
0=D^nU^nf=n!\, p_n(E)f \implies 0=Dp_n(E)f=p_n(E-2I)Df=p_n(E-2I)h
\]
Since the roots of $q_e(x)$ are $\{ 0,-1,\hdots ,-e\}$ and the roots of $p_n(x-2)$ are $\{ 2,3,\hdots ,n+1\}$ we see that
$\gcd_{\Q}(q_e(x),p_n(x-2))=1$. Since $q_e(E)h=p_n(E-2I)h=0$ we conclude that $h=0$, which gives a contradiction. 

Therefore, $\Omega\cap DB=\{ 0\}$.
\end{proof}

\begin{lemma}\label{critical} For all integers $n\ge 1$, $\krn U^n\cap {\rm im}\,D^n=\{ 0\}$.
\end{lemma}

\begin{proof} The case $n=1$ is proved in {\it Lemma\,\ref{ker-image}}. Assume that $n\ge 2$. 
Let $h\in \krn U^n\cap {\rm im}\,D^n$ and let $f\in B$ be such that $h=D^nf$. Then $U^nD^nf=0$.

Suppose that $D^nf\ne 0$. Let $m\in\N$ be minimal so that $U^mD^nf=0$, noting that $m\le n$. We have:
\[
U^{m-1}D^{m-1}(D^{n-m+1}f)\in\Omega
\]
By {\it Lemma\,\ref{subalgebra}} there exists $P(x,y)\in k[x,y]\cong k^{[2]}$ such that $U^{m-1}D^{m-1}=P(DU,E)$. Therefore:
\[
P(DU,E)(D^{n-m+1}f)\in\Omega
\]
Because $DU$ commutes with $E$ we can write 
\[
P(DU,E)(D^{n-m+1}f)=Dg+P(0,E)D^{n-m+1}f
\]
for some $g\in B$. Since $ED=D(E+2I)$ we see that:
\[
 P(0,E)D^{n-m+1}=D^{n-m+1}P(0,E+2(n-m+1)I)
 \]
 Therefore, 
 \[
 U^{m-1}D^{m-1}(D^{n-m+1}f)=Dg+D^{n-m+1}P(0,E+2(n-m+1)I)f\in \Omega\cap DB=\{ 0\}
 \]
 which implies $U^{m-1}D^nf=0$, contradicting minimality of $m$. So $h=D^nf=0$.
\end{proof}

Recall that $I_n=A\cap D^nB$ is the $n$-th image $A$-ideal for $D$.
\begin{lemma}\label{A-grading} Given $d\in\N$ let $A_d=A\cap\krn (E-dI)$. Then 
\begin{itemize}
\item [{\bf (a)}] $A=\bigoplus_{d\ge 0}A_d$ is an $\N$-grading.
\item [{\bf (b)}] Given nonzero $f\in A_d$, $d=\deg_Uf$.
\item [{\bf (c)}] Given $n\in\N$, $I_n=\bigoplus_{d\ge n}A_d$. 
\end{itemize}
\end{lemma}

\begin{proof} First observe that $E$ restricts to $A$: Given $f\in A$, we have $0=(E-2I)Df=DEf$. 

Given nonzero $f\in A$ let $d=\deg_Uf$. By {\it Lemma\,\ref{min-poly}}, $p_{d+1}(E)f=0$ so by {\it Lemma\,\ref{sum-decomp}} there exists a unique sequence $f_i\in\krn (E-iI)$, $0\le i\le d$, such that 
$f=\sum_{i=0}^df_i$. Therefore, $E^tf=\sum_{i=0}^di^tf_i$ for each $0\le t\le d$. It follows that the vector space spanned by $f,Ef,\hdots ,E^df$ is the same as that spanned by $f_0,f_1,\hdots ,f_d$ since the corresponding Vandermonde matrix is invertible. Since $E$ restricts to $A$, we conclude that $f_i\in A$ for each $i$. By uniqueness, we get $A=\bigoplus_{d\ge 0}A_d$, which is an $\N$-grading by restriction from $B'$. This proves part (a).

For part (b), let $f\in A_d$ be nonzero and let $n=\deg_Uf$. By equation (\ref{identity}) we have
\begin{eqnarray*}
0&=&DU^{n+1}f \\
&=& U^n(UD+(n+1)E-n(n+1)I)f \\
&=& (n+1)U^n(E-nI)f \\
&=& (n+1)U^n((d-n)f) \\
&=& (n+1)(d-n)U^nf
\end{eqnarray*}
which shows that $d=n$.

For part (c), define $A$-ideals $J_n=\bigoplus_{i\ge n}A_i$, $n\ge 0$. 
Given $d\ge 1$ and nonzero $f\in A_d$, part (a) shows that $Ef=df$. Equation (\ref{identity2}) thus gives $D^dU^df=d\, !p_d(d)f$ where $p_d(d)\ne 0$. 
It follows that $A_d\subset I_n$ when $d\ge n$, so $J_n\subseteq I_n$ for all $n\ge 0$. 
Given $h\in I_n$, part (a) shows that $h=h_0+\cdots h_{n-1}+g$ for some $h_i\in A_i$ and $g\in J_n$, so $h-g\in D^nB$. By part (b), $h-g\in\krn U^n$. 
Therefore, {\it Lemma\,\ref{critical}} shows $h-g\in {\rm im}D^n\cap\krn U^n=\{ 0\}$. So $I_n =J_n$. 
\end{proof}

Recall that $\mathcal{F}_n=\krn D^{n+1}$, $n\in\N\cup\{ -1\}$, are the degree modules for $D$.

\begin{lemma}\label{final-straw} Given an integer $n\ge 1$ and nonzero $f\in U^nD^n(\mathcal{F}_n)$, let $d=\deg_U(D^nf)$ (which is well-defined, since $D^nf\ne 0)$. 
Then $p_d(E+2nI)f=0$.
\end{lemma} 

\begin{proof} We have:
\[
0=p_d(E)(D^nf)=D^np_d(E+2nI)f \implies p_d(E+2nI)f\in\krn D^n
\]
In addition, $f=U^ng$ for some $g\in B$ which implies:
\[
p_d(E+2nI)f=p_d(E+2nI)U^ng=U^np_d(E)g \implies p_d(E+2nI)f\in {\rm im}\, U^n
\]
By {\it Lemma\,\ref{critical}}, $p_d(E+2nI)f=0$.
\end{proof}

\begin{lemma}\label{module-sum} Given $d,n\in\Z$ with $n\ge 1$, define $c_i=i(d+2n-i+1)$, $1\le i\le n$. 
\begin{itemize}
\item [{\bf (a)}] If $f\in\mathcal{F}_n\cap B_d$ then $U^nD^n(f)-c_1\cdots c_nf\in\mathcal{F}_{n-1}$. 
\item [{\bf (b)}] If $f\in (\mathcal{F}_n\setminus\mathcal{F}_{n-1})\cap B_d$ then $d+n\ge 0$. 
\item [{\bf (c)}] As $A_0$-modules: $\mathcal{F}_n=\mathcal{F}_{n-1}\oplus U^nD^n(\mathcal{F}_n)$
\item [{\bf (d)}] As $A_0$-modules:
\[
B=\bigoplus_{i\ge 0}U^iD^i(\mathcal{F}_i)
\]
\end{itemize}
\end{lemma}

\begin{proof}  Part (a). Define $\varphi :\mathcal{F}_n\to B$ by $\varphi =U^nD^n$. By definition, $D^ng\in A$ for each $g\in\mathcal{F}_n$. Equation (\ref{identity2}) implies:
\[
D^{n+1}(U^nD^ng)=DD^nU^n(D^ng)=D(n!\,p_n(E)(D^ng)=n!\,p_n(E+2I)(D^{n+1}g)=0
\]
Therefore, $U^nD^ng\in\mathcal{F}_n$, so $\varphi (\mathcal{F}_n)\subset\mathcal{F}_n$. Suppose that $g\in\krn\varphi$. 
Then $U^nD^ng=0$ implies $D^ng\in\krn U^n\cap {\rm im}\,D^n=\{ 0\}$, by {\it Lemma\,\ref{critical}}, meaning $g\in\mathcal{F}_{n-1}$. 
So $\krn\varphi\subseteq\mathcal{F}_{n-1}$, and the reverse inclusion is clear. 
Therefore, $\krn\varphi=\mathcal{F}_{n-1}$ and the sequence
\[
0\to \mathcal{F}_{n-1}\to \mathcal{F}_n\xrightarrow{\varphi} U^nD^n\mathcal{F}_n\to 0
\]
is exact. 
We have $D^nf\in A_e$ for $e=d+2n$. By \cite{Freudenburg2022}, Lemma 3.5, $D^nU^n(D^nf)=c_1\cdots c_nD^nf$ for $c_i=i(e-i+1)=i(d+2n-i+1)$. It follows that
\[
\varphi^2(f)=c_1\cdots c_n\varphi (f) \implies \varphi (f)-c_1\cdots c_nf\in\krn\varphi=\mathcal{F}_{n-1}
\]
This proves part (a).

Part (b). By part (a), $f\notin\mathcal{F}_{n-1}=\krn\varphi$ implies $D^nf\in A_e\setminus\{ 0\}$. 
By {\it Theorem\,\ref{main2}(b)}, $\deg_U(D^nf)=e$. Since $U^n(D^nf)\ne 0$ it follows that $n\le e=d+2n$. This proves part (b).

Part (c). Let $W=\mathcal{F}_{n-1}+U^nD^n(\mathcal{F}_n)\subseteq\mathcal{F}_n$ and let 
$f\in \mathcal{F}_n\cap B_d$ be given.
If $\varphi (f)=0$ then $f\in\mathcal{F}_{n-1}\subset W$. If $f\notin\krn\varphi =\mathcal{F}_{n-1}$ then 
part (b) implies $e-i\ge e-n=d+2n-n=d+n\ge 0$ so $c_i\ne 0$ for each $1\le i\le n$. 
Therefore, $c_1\cdots c_n\ne 0$. By part (a) we see that $f\in W$. So 
every homogeneous element of $\mathcal{F}_n$ lies in $W$. 
Since $\mathcal{F}_n$ is a graded module, 
we see that $\mathcal{F}_n=W$. 
By symmetry of $D$ with $U$, {\it Lemma\,\ref{critical}} implies
\[
\mathcal{F}_{n-1}\cap U^nD^n\mathcal{F}_n\subseteq \krn D^n\cap {\rm im}\,U^n=\{ 0\}
\]
and part (c) is proved. 

Part (d). This follows immediately from part (c). 
\end{proof}

\begin{lemma}\label{generate} $B=\bigoplus_{n\in\Z}B_n$ is a $\Z$-grading.
\end{lemma}

\begin{proof} 
Given nonzero $f\in B$, {\it Lemma\,\ref{module-sum}} shows that there exists $r\in\N$ with $r\ge 1$ and $f_i\in U^iD^i\mathcal{F}_i$, $1\le i\le r$, such that $f=\sum_if_i$. 
{\it Lemma\,\ref{final-straw}} shows that, for each $i$, there exist $d_i,n_i\in\Z$ such that $f_i\in\krn p_{d_i}(E+n_iI)$. {\it Lemma\,\ref{sum-decomp}} shows that $f_i\in B'$ for each $i$.
Therefore, $f\in B'$ and $B=B'$. 
The fact that $B=\bigoplus_{n\in\Z}B_n$ is a $\Z$-grading follows from {\it Lemma\,\ref{sum-decomp}}. 
\end{proof}
This completes the proofs for {\it Theorem\,\ref{main1}} and {\it Theorem\,\ref{main2}}. 

\subsection{Proof for Theorem\,\ref{main3}}

The inclusion $A+\mathfrak{p}\subseteq D^{-1}(\mathfrak{p})$ is clear from the hypotheses. For the reverse inclusion, we show by induction on $n\ge -1$ that 
$D^{-1}(\mathfrak{p}\cap\mathcal{F}_n)\subset A+\mathfrak{p}$
where $\mathcal{F}_{-1}:=\{ 0\}$.
Since $\mathcal{F}_{-1}=\{ 0\}$, a basis for induction holds. 

Assume that $D^{-1}(\mathfrak{p}\cap\mathcal{F}_n)\subset A+\mathfrak{p}$ for some $n\ge -1$, and let homogeneous $f\in D^{-1}(\mathfrak{p}\cap\mathcal{F}_{n+1})$ be given. 
Then $f\in\mathcal{F}_{n+2}$. If $f\in\mathcal{F}_{n+1}$ then 
$f\in A+\mathfrak{p}$ by the inductive hypothesis, so we may assume $f\notin\mathcal{F}_{n+1}$. 
Let $h=U^{n+2}D^{n+2}(f)$. Since $Df\in\mathfrak{p}$ and $\mathfrak{p}$ is $(D,U)$-invariant, we see that $h\in\mathfrak{p}$. 
By {\it Lemma\,\ref{module-sum}}, there exists $g\in\mathcal{F}_{n+1}$ and positive $c\in\Z$ such that 
$cf=g+h$. Since $Dg=cDf-Dh\in\mathfrak{p}$, the inductive hypothesis implies that $g\in A+\mathfrak{p}$. Therefore, $f\in A+\mathfrak{p}$. 
Since $A+\mathfrak{p}$ is a graded $A$-module, we conclude that $D^{-1}(\mathfrak{p}\cap\mathcal{F}_n)\subset A+\mathfrak{p}$. 

By induction, it follows that $D^{-1}(\mathfrak{p}\cap\mathcal{F}_n)\subset A+\mathfrak{p}$ for all $n\ge -1$. Consequently, $D^{-1}(\mathfrak{p})\subseteq A+\mathfrak{p}$ and part (a) is proved. 

From part (a), it follows immediately that $\krn (\pi (D))=\pi (\krn D)$. The fact that 
\[
    \krn (\pi (U))=\pi (\krn U)
\]
follows by symmetry. 
$\hfill \qed$

\subsection{Proof for Lemma \,\ref{useful2}}
Part (a). Assume that $f\in A_1$ is nonzero. Then $f$ is a local slice for
$U$ in $\krn(D)$, so by {\it Lemma\,\ref{useful}} we have 
${\rm frac}(R)={\rm frac}(B)$.

Part (b). Assume that $f\in A_2$ is nonzero. Then $g:=Uf\in B_0=\krn
E$ and $Dg$ is a nonzero multiple of $f$, so $g$ is a local
slice for $D$ in $\krn(E)$. By {\it Lemma\,\ref{useful}}, 
if $S'=k[\krn D,\krn E]$ then ${\rm frac}(S')={\rm frac}(B)$.
Since $S'\subset S$ we have ${\rm frac}(S)={\rm frac}(B)$.

Part (c). If $f\in A_1$ is nonzero then $g=Uf$ satisfies 
$\deg_U(g)=0$ and $\deg_D(g)=1$. 
Conversely, assume that $g\in B$ satisfies $\deg_U(g)=0$ and $\deg_D(g)=1$. 
Let $g=\sum_ig_i$ be the decomposition of $g$ into homogeneous summands. Since $g\in\Omega$ we see from {\it Theorem\,\ref{main2}(c)} that $i\le -1$ and $\deg_D(g_i)= \lvert i \rvert$ whenever $g_i\ne 0$. So the only possibility is $g=g_{-1}+g_0$ where $g_{-1}\ne 0$. We thus have $Dg=Dg_{-1}\in A_1$ and $Dg\ne 0$.

Part (d). If $f\in A_2$ is nonzero then $g=Uf$ satisfies 
$\deg_D(g)=\deg_U(g)=1$. 
Conversely, assume that $g\in B$ satisfies $\deg_D(g)=\deg_U(g)=1$. 
Let $g=\sum_ig_i$ be the decomposition of $g$ into homogeneous summands. Since $D$ is homogeneous, the degree module $\mathcal{F}_1=\krn D^2$ is a graded module. So $g_i\in\mathcal{F}_1$ for each $i$, i.e., 
$\deg_Dg_i\le 1$ for each $i$. By symmetry, $\deg_U(g_i)\le 1$ for each $i$. We may assume that, if $g_i\ne 0$, then either $Dg_i\ne 0$ or $Ug_i\ne 0$. 

If $Dg_i\ne 0$ and $Ug_i=0$ for some $i$ then we are in the situation of part (c). Therefore, in this case, $A_1\ne\{ 0\}$, so $A_2\ne \{ 0\}$. The same reasoning works if $Dg_i=0$ and $Ug_i\ne 0$.

Assume that $Dg_i\ne 0$ and $Ug_i\ne 0$ whenever $g_i\ne 0$. 
 Since $\deg D=2$, {\it Theorem\,\ref{main2}(c)} shows:
\[
Dg_i\in A_{i+2}\setminus\{ 0\} \implies i+2\ge 1 \implies i\ge -1
\]
Likewise, $\deg U=-2$ and by symmetry we have:
\[
Ug_i\in \Omega_{i-2}\setminus\{ 0\} \implies i-2\le -1 \implies i\le 1
\]
Therefore, $g_i\ne 0$ implies $-1\le i\le 1$, and we have $g=g_{-1}+g_0+g_1$. 

Assume that $g_{-1}\ne 0$. Then:
\[
A_{-1}=\{ 0\} \implies Dg_{-1}\ne 0 \implies A_1\ne \{ 0\} \implies A_2\ne \{ 0\}
\]
If $g_1\ne 0$ then $\Omega_{-2}\ne \{ 0\}$ by symmetry, which implies $A_2\ne \{ 0\}$ by symmetry. 

If $g_{-1}=g_1=0$ then $g\in B_0$, which implies $Dg\in A_2$ and $Dg\ne 0$ so $A_2\ne \{ 0\}$. 
\hfill \qed

\section*{Funding}

Rafael Andrist was supported by the European Union (ERC Advanced grant HPDR, 101053085 to Franc Forstneri\v{c}). Gaofeng Huang and Frank Kutzschebauch were partially supported by Schweizerischer Nationalfonds (SNSF) grant 200021-207335; and Jan Draisma was partially supported by SNSF grant 200021-227864.

\section*{Conflict of Interest}
The authors have no relevant competing interest to disclose.



\begin{bibdiv}
\begin{biblist}

\bib{MR1185588}{article}{
   author={Anders\'en, Erik},
   author={Lempert, L\'aszl\'o},
   title={On the group of holomorphic automorphisms of ${\bf C}^n$},
   journal={Invent. Math.},
   volume={110},
   date={1992},
   number={2},
   pages={371--388},
   issn={0020-9910},
   review={\MR{1185588}},
   doi={10.1007/BF01231337},
}

\bib{MR4305975}{article}{
   author={Andrist, Rafael B.},
   title={The density property for Calogero-Moser spaces},
   journal={Proc. Amer. Math. Soc.},
   volume={149},
   date={2021},
   number={10},
   pages={4207--4218},
   issn={0002-9939},
   review={\MR{4305975}},
   doi={10.1090/proc/15457},
}

\bib{Andrist.Freudenburg.Huang.Kutzschebauch.Schott}{article}{
    author={Andrist, Rafael B.},
    author={Freudenburg, Gene},
    author={Huang, Gaofeng},
    author={Kutzschebauch, Frank},
    author={Schott, Josua},
    title={A Criterion for the Density Property of Stein Manifolds},
    journal = {Michigan Mathematical Journal},
    publisher = {University of Michigan, Department of Mathematics},
    pages = {1 -- 24},
    keywords = {14R20, 20G35, 32M05, 32M17, 32M25, 32Q56},
    date = {2025},
    doi = {10.1307/mmj/20236469},
    URL = {https://doi.org/10.1307/mmj/20236469},
}

\bib{Andrist:2024aa}{article}{
    author={Andrist, Rafael B.},
    author={Kutzschebauch, Frank},
    year={2024},
    journal={Beitr\"age zur Algebra und Geometrie / Contributions to Algebra and Geometry},
    title={Algebraic overshear density property},
    doi={10.1007/s13366-023-00729-4},
}


\bib{MR3039680}{article}{
   author={Arzhantsev, I.},
   author={Flenner, H.},
   author={Kaliman, S.},
   author={Kutzschebauch, F.},
   author={Zaidenberg, M.},
   title={Flexible varieties and automorphism groups},
   journal={Duke Math. J.},
   volume={162},
   date={2013},
   number={4},
   pages={767--823},
   issn={0012-7094},
   review={\MR{3039680}},
   doi={10.1215/00127094-2080132},
}

\bib{MR3049288}{article}{
   author={Arzhantsev, Ivan},
   author={Liendo, Alvaro},
   title={Polyhedral divisors and $\mathrm{SL}_2$-actions on affine
   $\mathbb{T}$-varieties},
   journal={Michigan Math. J.},
   volume={61},
   date={2012},
   number={4},
   pages={731--762},
   issn={0026-2285},
   review={\MR{3049288}},
   doi={10.1307/mmj/1353098511},
}

\bib{MR1785579}{article}{
   author={Berest, Yuri},
   author={Wilson, George},
   title={Automorphisms and ideals of the Weyl algebra},
   journal={Math. Ann.},
   volume={318},
   date={2000},
   number={1},
   pages={127--147},
   issn={0025-5831},
   review={\MR{1785579}},
   doi={10.1007/s002080000115},
}

\bib{MR2739794}{article}{
   author={Bielawski, Roger},
   author={Pidstrygach, Victor},
   title={On the symplectic structure of instanton moduli spaces},
   journal={Adv. Math.},
   volume={226},
   date={2011},
   number={3},
   pages={2796--2824},
   issn={0001-8708},
   review={\MR{2739794}},
   doi={10.1016/j.aim.2010.10.001},
}


\bib{MR4418718}{article}{
   author={Chen, Xiaojun},
   author={Eshmatov, Farkhod},
   author={Eshmatov, Alimjon},
   author={Tikaradze, Akaki},
   title={On transitive action on quiver varieties},
   journal={Int. Math. Res. Not. IMRN},
   date={2022},
   number={10},
   pages={7694--7728},
   issn={1073-7928},
   review={\MR{4418718}},
   doi={10.1093/imrn/rnaa339},
}

\bib{MR1834739}{article}{
   author={Crawley-Boevey, William},
   title={Geometry of the moment map for representations of quivers},
   journal={Compositio Math.},
   volume={126},
   date={2001},
   number={3},
   pages={257--293},
   issn={0010-437X},
   review={\MR{1834739}},
   doi={10.1023/A:1017558904030},
}



\bib{MR2718937}{article}{
   author={Donzelli, F.},
   author={Dvorsky, A.},
   author={Kaliman, S.},
   title={Algebraic density property of homogeneous spaces},
   journal={Transform. Groups},
   volume={15},
   date={2010},
   number={3},
   pages={551--576},
   issn={1083-4362},
   review={\MR{2718937}},
   doi={10.1007/s00031-010-9091-8},
}

\bib{MR1881922}{article}{
   author={Etingof, Pavel},
   author={Ginzburg, Victor},
   title={Symplectic reflection algebras, Calogero-Moser space, and deformed
   Harish-Chandra homomorphism},
   journal={Invent. Math.},
   volume={147},
   date={2002},
   number={2},
   pages={243--348},
   issn={0020-9910},
   review={\MR{1881922}},
   doi={10.1007/s002220100171},
}

\bib{MR4440754}{article}{
   author={Forstneri\v{c}, F.},
   author={Kutzschebauch, F.},
   title={The first thirty years of Anders\'{e}n-Lempert theory},
   journal={Anal. Math.},
   volume={48},
   date={2022},
   number={2},
   pages={489--544},
   issn={0133-3852},
   review={\MR{4440754}},
   doi={10.1007/s10476-022-0130-1},
}

\bib{MR1213106}{article}{
   author={Forstneri\v c, Franc},
   author={Rosay, Jean-Pierre},
   title={Approximation of biholomorphic mappings by automorphisms of ${\bf
   C}^n$},
   journal={Invent. Math.},
   volume={112},
   date={1993},
   number={2},
   pages={323--349},
   issn={0020-9910},
   review={\MR{1213106}},
   doi={10.1007/BF01232438},
}

\bib{MR3700208}{book}{
   author={Freudenburg, Gene},
   title={Algebraic theory of locally nilpotent derivations},
   series={Encyclopaedia of Mathematical Sciences},
   volume={136},
   edition={2},
   note={Invariant Theory and Algebraic Transformation Groups, VII},
   publisher={Springer-Verlag, Berlin},
   date={2017},
   pages={xxii+319},
   isbn={978-3-662-55348-0},
   isbn={978-3-662-55350-3},
   review={\MR{3700208}},
   doi={10.1007/978-3-662-55350-3},
}

\bib{Freudenburg2022}{article}{
   author={Freudenburg, Gene},
   title={Actions of $SL_2(k)$ on affine $k$-domains and fundamental pairs},
   journal={Transform. Groups},
   volume={29},
   date={2024},
   number={3},
   pages={959--1003},
   issn={1083-4362},
   review={\MR{4788020}},
   doi={10.1007/s00031-022-09750-8},
}

\bib{MR0761664}{article}{
   author={Gibbons, John},
   author={Hermsen, Theo},
   title={A generalisation of the Calogero-Moser system},
   journal={Phys. D},
   volume={11},
   date={1984},
   number={3},
   pages={337--348},
   issn={0167-2789},
   review={\MR{0761664}},
   doi={10.1016/0167-2789(84)90015-0},
}


\bib{MR2385667}{article}{
   author={Kaliman, Shulim},
   author={Kutzschebauch, Frank},
   title={Criteria for the density property of complex manifolds},
   journal={Invent. Math.},
   volume={172},
   date={2008},
   number={1},
   pages={71--87},
   issn={0020-9910},
   review={\MR{2385667}},
   doi={10.1007/s00222-007-0094-6},
}


\bib{MR2660454}{article}{
   author={Kaliman, Shulim},
   author={Kutzschebauch, Frank},
   title={Algebraic volume density property of affine algebraic manifolds},
   journal={Invent. Math.},
   volume={181},
   date={2010},
   number={3},
   pages={605--647},
   issn={0020-9910},
   review={\MR{2660454}},
   doi={10.1007/s00222-010-0255-x},
}


\bib{MR3492044}{article}{
   author={Kaliman, Sh.},
   author={Kutzschebauch, F.},
   title={On algebraic volume density property},
   journal={Transform. Groups},
   volume={21},
   date={2016},
   number={2},
   pages={451--478},
   issn={1083-4362},
   review={\MR{3492044}},
   doi={10.1007/s00031-015-9360-7},
}


\bib{MR3623226}{article}{
   author={Kaliman, Shulim},
   author={Kutzschebauch, Frank},
   title={Algebraic (volume) density property for affine homogeneous spaces},
   journal={Math. Ann.},
   volume={367},
   date={2017},
   number={3-4},
   pages={1311--1332},
   issn={0025-5831},
   review={\MR{3623226}},
   doi={10.1007/s00208-016-1451-9},
}



\bib{MR0974333}{article}{
   author={Howe, Roger},
   title={{\it The classical groups} and invariants of binary forms},
   conference={
      title={The mathematical heritage of Hermann Weyl},
      address={Durham, NC},
      date={1987},
   },
   book={
      series={Proc. Sympos. Pure Math.},
      volume={48},
      publisher={Amer. Math. Soc., Providence, RI},
   },
   isbn={0-8218-1482-6},
   date={1988},
   pages={133--166},
   review={\MR{0974333}},
   doi={10.1090/pspum/048/974333},
}

\bib{MR0929658}{article}{
   author={Rosay, Jean-Pierre},
   author={Rudin, Walter},
   title={Holomorphic maps from ${\bf C}^n$ to ${\bf C}^n$},
   journal={Trans. Amer. Math. Soc.},
   volume={310},
   date={1988},
   number={1},
   pages={47--86},
   issn={0002-9947},
   review={\MR{0929658}},
   doi={10.2307/2001110},
}


\bib{MR1829353}{article}{
   author={Varolin, Dror},
   title={The density property for complex manifolds and geometric
   structures},
   journal={J. Geom. Anal.},
   volume={11},
   date={2001},
   number={1},
   pages={135--160},
   issn={1050-6926},
   review={\MR{1829353}},
   doi={10.1007/BF02921959},
}


\bib{MR1626461}{article}{
   author={Wilson, George},
   title={Collisions of Calogero-Moser particles and an adelic Grassmannian},
   note={With an appendix by I. G. Macdonald},
   journal={Invent. Math.},
   volume={133},
   date={1998},
   number={1},
   pages={1--41},
   issn={0020-9910},
   review={\MR{1626461}},
   doi={10.1007/s002220050237},
}

\bib{MR3813594}{article}{
   author={Zentner, Raphael},
   title={Integer homology 3-spheres admit irreducible representations in
   ${\rm SL}(2,\mathbb{C})$},
   journal={Duke Math. J.},
   volume={167},
   date={2018},
   number={9},
   pages={1643--1712},
   issn={0012-7094},
   review={\MR{3813594}},
   doi={10.1215/00127094-2018-0004},
}

\end{biblist}
\end{bibdiv}
\end{document}